\documentclass[reqno]{amsart}
\usepackage{amsmath,amsfonts,amsthm,amssymb}
\usepackage{enumitem}

\usepackage[normalem]{ulem}

\usepackage{bbm}
\usepackage{hyperref}

\allowdisplaybreaks   

\theoremstyle{plain}
\newtheorem{de}{Definition}[section]
\newtheorem{lem}[de]{Lemma}
\newtheorem{prop}[de]{Proposition}
\newtheorem{cor}[de]{Corollary}
\newtheorem{thm}[de]{Theorem}

\theoremstyle{definition}
\newtheorem{rem}[de]{Remark}

\numberwithin{equation}{section}

\newcommand{\eul}{e}
\newcommand{\imu}{\mathrm{i}}
\newcommand{\dd}{\,\mathrm{d}}
\renewcommand{\epsilon}{\varepsilon}

\renewcommand{\Im}{\operatorname{Im}}
\renewcommand{\Re}{\operatorname{Re}}
\newcommand{\supp}{\operatorname{supp}}

\newcommand{\R}{{\mathbb{R}}}
\newcommand{\N}{{\mathbb{N}}}

\newcommand{\Z}{{\mathbb{Z}}}

\newcommand{\PP}{{\mathbb{P}}}

\newcommand{\cL}{{\mathcal{L}}}
\newcommand{\cF}{{\mathcal{F}}}
\newcommand{\cA}{{\mathcal{A}}}

\newcommand{\om}{\omega}

\newcommand{\ph}{\varphi}

\begin{document}
 \title[Almost sure local wellposedness and scattering for the energy-critical cubic NLS]{Almost sure local wellposedness and scattering for the energy-critical cubic nonlinear Schr{\"o}dinger equation with supercritical data }
 
 \author[M.~Spitz]{Martin Spitz}
 \address[M.~Spitz]{Fakult\"at f\"ur Mathematik, Universit\"at Bielefeld, Postfach 10 01 31, 33501
  Bielefeld, Germany}
\email{mspitz@math.uni-bielefeld.de}

\begin{abstract}
	We study the cubic defocusing nonlinear Schr{\"o}dinger equation on $\R^4$ with supercritical initial data. For randomized initial data in $H^s(\R^4)$, we prove almost sure local wellposedness for $\frac{1}{7} < s < 1$ and almost sure scattering for $\frac{5}{7} < s < 1$. The randomization is based on a unit-scale decomposition in frequency space, a decomposition in the angular variable, and -- for the almost sure scattering result -- an additional unit-scale decomposition in physical space. We employ new probabilistic estimates for the linear Schr{\"o}dinger flow with randomized data, where we effectively combine the advantages of the different decompositions.
\end{abstract}

\maketitle

\section{Introduction and main results}
We consider the cubic defocusing nonlinear Schr{\"o}dinger equation in four space dimensions, i.e.
\begin{equation}
	\label{eq:CubicNLSDef}
	\begin{aligned}
		\imu \partial_t u + \Delta u &= |u|^2 u \qquad \text{on } \R \times \R^4, \\
		u(0) &= f
	\end{aligned}
\end{equation}
with $f \in H^s(\R^4)$. The equation is called defocusing because of the plus sign in front of $|u|^2 u$. With a minus sign in front of the nonlinearity, the corresponding equation is called focusing. The energy
\begin{align*}
	E(u(t)) = \int_{\R^4} \frac{1}{2} |\nabla u(t,x)|^2 + \frac{1}{4} |u(t,x)|^4 \dd x
\end{align*}
is conserved for sufficiently regular solutions of~\eqref{eq:CubicNLSDef}. Since the rescaling
\begin{align*}
	u(t,x) \mapsto \lambda u(\lambda^2 t, \lambda x)
\end{align*}
leaves both equation~\eqref{eq:CubicNLSDef} and the energy invariant, problem~\eqref{eq:CubicNLSDef} is called energy-critical.

The purpose of this article is to investigate the local wellposedness as well as the long-time behavior of solutions of~\eqref{eq:CubicNLSDef} with initial data below the scaling critical regularity, i.e. with $f \in H^s(\R^4)$ with $s < 1$. We prove that~\eqref{eq:CubicNLSDef} is almost surely locally wellposed for initial data $f \in H^s(\R^4)$ with $\frac{1}{7} < s < 1$ which is randomized with respect to a certain decomposition both in frequency space and in the angular variable. Decomposing the initial data $f \in H^s(\R^4)$ additionally in physical space, we show that randomizing with respect to the resulting decomposition leads to global scattering solutions of~\eqref{eq:CubicNLSDef} almost surely for $\frac{5}{7} < s < 1$.

The energy-critical defocusing nonlinear Schr{\"o}dinger equation (NLS) was subject of extensive research, which culminated in~\cite{RV2007} and~\cite{V2012} in four space dimensions, building on the fundamental work~\cite{CKSTT2008} in $\R^3$. In~\cite{RV2007, V2012} it was shown that solutions of~\eqref{eq:CubicNLSDef} in the energy space exist globally and scatter. More precisely, for every $f \in \dot{H}^1(\R^4)$ there is a unique global solution $u$ of~\eqref{eq:CubicNLSDef} which scatters as $t \rightarrow \pm \infty$ and satisfies
\begin{equation}
	\label{eq:SpaceTimeBoundByEnergy}
	\|\nabla u\|_{L^3_t L^3_x} \leq L(E(u_0)),
\end{equation}
where $L \colon [0,\infty) \rightarrow [0,\infty)$ is a non-decreasing function, see Theorem~1.1 and Lemma~3.6 in~\cite{RV2007}. We refer to~\cite{CKSTT2008, RV2007, V2012} and the references therein for the historical development and prior work on the energy-critical defocusing NLS.

Below the scaling-critical exponent $s_c = 1$, problem~\eqref{eq:CubicNLSDef} is known to be illposed, see~\cite{CCT2003}. However, how generic is the ill behavior of~\eqref{eq:CubicNLSDef} below the scaling critical regularity? Can one still identify large sets of initial data for which unique local solutions, global solutions, and scattering solutions exist? One way to investigate this question is to use randomization.

	\subsection{Randomization}
	\label{subsec:Randomization}
	
	Starting from the seminal works~\cite{B1994, B1996} and~\cite{BT2008I,BT2008II}, large interest in random dispersive partial differential equations and a vast body of literature has developed over the last years. Therefore, we restrict our discussion to results which are closest to this work.
	
	In~\cite{BOP2015I, BOP2015II} almost sure local wellposedness of the cubic NLS on $\R^d$ with $d \geq 3$ for suitably randomized initial data $f \in H^s(\R^d)$ was shown, where $\frac{3}{5} < s < 1$ in the case $d = 4$. We also refer to~\cite{BOP2019,SSW2021} for improvements in the case of the cubic NLS on $\R^3$, to~\cite{B2019} for the energy-critical NLS on $\R^3$, and to~\cite{OOP2019} for the energy-critical NLS on $\R^d$ with $d = 5,6$. The regularity threshold for almost sure local wellposedness of the cubic NLS in $\R^4$ was lowered in~\cite{DLM2019} to $\frac{1}{3} < s < 1$.
	
	As a consequence of the local theory, conditional global results and global results on sets with positive probability were provided in~\cite{BOP2015I} and~\cite{B2019} for the energy-critical NLS. Apart from the small data results, the global theory is more delicate. The first advances in this direction were made for the wave equation. To be more precise, almost sure global existence for the defocusing energy-critical wave equation was shown in~\cite{P2017} for $d = 4$ and $d = 5$ and in~\cite{OP2016} for $d = 3$.
	
	The first almost sure scattering result for an energy-critical dispersive equation (besides the small data theory) was given in~\cite{DLM2020} for the defocusing energy-critical wave equation in $\R^4$ with initial data which is radially symmetric before the randomization. See also~\cite[Appendix~A]{DLM2019} for an improvement and~\cite{Br2020} for the case $d = 3$. Building upon ideas from~\cite{DLM2020}, almost sure scattering for the solutions of the defocusing energy-critical NLS in $\R^4$ with randomized radially symmetric initial data from $H^s(\R^4)$ was proven in~\cite{KMV2019} for $\frac{5}{6} < s < 1$ and then improved in~\cite{DLM2019} to $\frac{1}{2} < s < 1$. Very recently, almost sure scattering for randomized radially symmetric initial data was also established for the cubic defocusing NLS on $\R^3$ in~\cite{C2021} and~\cite{SSW2021}. While the preceding almost sure scattering results all rely on the radial symmetry of the original data, the assumption of radial symmetry for the initial data before randomization was removed in~\cite{Br2019} for the defocusing energy-critical wave equation in $\R^4$.
	
	The main motivation of this work is to prove almost sure scattering for the defocusing energy-critical NLS in $\R^4$ without the radial symmetry assumption for the original initial data. As a byproduct, we improve the regularity threshold for almost sure local wellposedness of~\eqref{eq:CubicNLSDef} to $s > \frac{1}{7}$. However, we point out that we use a different randomization technique than the aforementioned works.
	
	In fact, all the above references cited in the context of almost sure local and global wellposedness of the energy-critical wave equation and NLS on the full-space except~\cite{Br2019} and~\cite{Br2020} employ a \emph{Wiener randomization}, i.e. a randomization based on a unit-scale decomposition of frequency space (see Subsection~\ref{subsec:UnitScaleDecompositionFreq} below for details). Nevertheless, there are other approaches. Considering the wave equation on compact manifolds with supercritical scaling, a randomization with respect to the decomposition in an orthonormal basis of eigenfunctions of the Laplacian was used in~\cite{BT2008I, BT2008II}. Similar techniques were then also employed on the full-space, see e.g.~\cite{D2012, dS2013, PRT2014} and references therein, before the Wiener randomization was introduced in~\cite{BOP2015I, BOP2015II} and~\cite{LM2014} and proved to be very successful on $\R^d$. The Wiener randomization gives access to a unit-scale Bernstein estimate, improving the Strichartz estimates from the deterministic setting.
	
	More recently, a randomization with respect to a unit-scale decomposition in physical space was introduced in~\cite{M2019} (see Subsection~\ref{subsec:UnitScaleDecompositionPhys} below for details) to improve almost surely the known uniqueness results for the final-state problem of the mass-subcritical NLS in $L^2$. Roughly speaking, this randomization allows to employ the dispersive estimate for solutions of the linear Schr{\"o}dinger equation although the data only belongs to $L^2$. These techniques were further refined in~\cite{NY2019}.
	
	In~\cite{BK2019} a randomization with respect to a decomposition in the angular variable was introduced (see Subsection~\ref{subsec:DecompositionGoodFrame} below for details). It was combined with a randomization in the radial variable and a Wiener randomization to show probabilistic global wellposedness of a wave maps type nonlinear wave equation for scaling supercritical data.
	
	 The interplay of the physical-space randomization from~\cite{M2019} and the angular randomization from~\cite{BK2019} was used in~\cite{S2020} to solve the final-state problem for the Zakharov system in $\R^3$ almost surely.
	 
	 Combinations of different randomization techniques were not only used in~\cite{BK2019} and~\cite{S2020}. In~\cite{Br2019}, where almost sure scattering for the defocusing energy-critical wave equation without the radial symmetry assumption was shown, the Wiener randomization was merged with the physical-space randomization to a microlocal randomization.
	 
	 For the defocusing energy-critical nonlinear Schr{\"o}dinger equation, we further add the angular randomization to prove almost sure scattering without assuming that the original data is radially symmetric. To lower the regularity threshold for the probabilistic local wellposedness theory, it is enough to combine the Wiener randomization with the angular randomization.
	 
	 Before we give the precise formulation of our results, we provide the details of these randomization procedures.

\subsection{Unit scale decomposition in frequency space}
\label{subsec:UnitScaleDecompositionFreq}
We fix a non-negative bump function $\phi \in C_c^\infty(\R^4)$ such that $\phi(x) = 1$ for $|x| \leq 1$ and $\phi(x) = 0$ for $|x| \geq 2$. We set
\begin{align*}
	\psi_j(\xi) = \frac{\phi(\xi - j)}{\sum_{m \in \Z^4} \phi(\xi - m)}
\end{align*}
for all $\xi \in \R^4$ and $j \in \Z^4$, which yields a smooth partition of unity $\{\psi_j \colon j \in \Z^4\}$. Introducing the operators $P_j$ by 
\begin{equation}
\label{eq:DefPk}
	P_j f = \cF^{-1}(\psi_j \hat{f})
\end{equation}
for all $j \in \Z^4$, where $\hat{f} = \cF f$ denotes the Fourier transform of $f$, we thus obtain the decomposition
\begin{equation}
	\label{eq:UnitScaleDecFreqSpace}
	f = \sum_{j \in \Z^4} P_j f
\end{equation}
for all $f \in L^2(\R^4)$. Randomizing with respect to this unit scale decomposition in frequency space yields the Wiener randomization, see e.g.~\cite{BOP2015I, BOP2015II, LM2014}.

\subsection{Unit scale decomposition in physical space}
\label{subsec:UnitScaleDecompositionPhys}	
To decompose a function in physical space, we employ the same partition of unity as above. However, in order to facilitate the distinction between the decomposition in frequency and in physical space, we use a different notation. Set
\begin{equation}
\label{eq:DefPhysicalSpacePartition}
	\ph_i = \psi_i
\end{equation}
for all $i \in \Z^4$. We then get
\begin{equation}
	\label{eq:UnitScaleDecPhysSpace}
	f = \sum_{i \in \Z^4} \ph_i f
\end{equation}
for every $f \in L^2(\R^4)$. The physical-space randomization means the randomization with respect to this decomposition, see~\cite{M2019, NY2019, S2020}.
	
\subsection{Decomposition in the angular variable with respect to a good frame}
\label{subsec:DecompositionGoodFrame}
	The last decomposition needs some preparation. We closely follow~\cite{BK2019, S2020} in our presentation.
	
	Recall that the spherical harmonics of degree $k$, i.e. the restriction of the homogeneous harmonic polynomials of degree $k$ to the sphere $S^3$, are the eigenfunctions of the Laplacian on the sphere. The space $E_k$ consisting of these functions has dimension
	\begin{align*}
		N_k = \binom{k+3}{3} - \binom{k+1}{3} = (k+1)^2.
	\end{align*}
	We now fix an orthonormal frame
	\begin{align*}
		\{b_{k,l} \in L^2(S^3) \colon l \in \{1, \ldots, N_k\}, k \in \N_0\}
	\end{align*}
	of $L^2(S^3)$, consisting of eigenfunctions of $\Delta_{S^3}$, such that there is a constant $C > 0$ with
	\begin{equation}
		\label{eq:BoundednessGoodFrame}
		\|b_{k,l}\|_{L^q(S^3)} \leq \begin{cases} 
										C \sqrt{q} \qquad &\text{if } q < \infty, \\
										C \sqrt{\log(k)} &\text{if } q = \infty
									\end{cases}
	\end{equation}		
	for all $l \in \{1, \ldots, N_k\}$, $k \in \N$, and $q \in [2,\infty]$. The existence of such a frame is a consequence of Th{\'e}or{\`e}me~6 and Proposition~3.2 in~\cite{BL2013}, see also~\cite[Theorem~1.1]{BK2019} and~\cite{BL2014}. Following~\cite{BK2019}, we call an orthonormal basis $\{b_{k,l} \colon l \in \{1, \ldots, N_k\}, k \in \N_0\}$ as above a \emph{good frame}.
	
	Now let $f \in L^2(\R^4)$. We first rescale the Littlewood-Paley blocks to frequency $1$ writing
	\begin{equation}
		\label{eq:DefgM}
		g^M = (P_M f)(M^{-1} \cdot)
	\end{equation}
	for every $M \in 2^{\Z}$, where $P_M$ denotes the standard Littlewood-Paley projector defined in Section~\ref{sec:NotationPrelim} below. After transition to polar coordinates we expand the Fourier transform of $g^M$ in the good frame, i.e. we write
	\begin{equation}
	\label{eq:gMInGoodFrame}
		\hat{g}^M(\rho \theta) = \sum_{k=0}^\infty \sum_{l = 1}^{N_k} \hat{c}^M_{k,l}(\rho) b_{k,l}(\theta),
	\end{equation}
	where every coefficient $\hat{c}^M_{k,l}$ is supported in $(\frac{1}{2},2)$. Theorem~3.10 in~\cite{SW71} thus yields
	\begin{align*}
		g^M(r \theta) = \sum_{k = 0}^\infty \sum_{l = 1}^{N_k} a_k r^{-1} b_{k,l}(\theta) \int_0^\infty \hat{c}_{k,l}^M(\rho) J_{k+1}(r \rho) \rho^2 \dd \rho,
	\end{align*}
	where $a_k = \imu^k (2 \pi)^{-2}$ and the Bessel function $J_{\mu}$ is given by 
	\begin{align*}
 		J_\mu(t) = \frac{(\frac{t}{2})^\mu}{\Gamma(\frac{2\mu+1}{2}) \Gamma(\frac{1}{2})}\int_{-1}^1 \eul^{\imu t s} (1-s^2)^{\frac{2\mu-1}{2}} \dd s
 	\end{align*}
 	for all $t > 0$ and $\mu > -\frac{1}{2}$. Moreover, Plancherel's theorem and~\eqref{eq:gMInGoodFrame} yield
	\begin{equation}
		\label{eq:PlancherelForgM}
		\|g^M\|_{L^2(\R^4)}^2 \sim \sum_{k = 0}^\infty \sum_{l = 1}^{N_k} \| \hat{c}^M_{k,l} \|_{L^2(\rho^3 \dd \rho)}^2.
	\end{equation}	 	
 	 Scaling back to frequency $M$, we obtain the representation
 	\begin{equation}
 	\label{eq:DecompositionAngular}
 		P_M f(r \theta) = g^M(M r \theta) = \sum_{k = 0}^\infty \sum_{l = 1}^{N_k} a_k M^{-1} r^{-1} b_{k,l}(\theta) \int_0^\infty \hat{c}_{k,l}^M(\rho) J_{k+1}(M r \rho) \rho^2 \dd \rho
 	\end{equation}
 	of $P_M f$ in polar coordinates. Randomizing with respect to this decomposition yields an angular randomization, see~\cite{BK2019, S2020}.
 	
 	\subsection{Definition of the randomizations}
 	\label{subsec:DefRandomizations}
 	
 	 Combining the angular decomposition in the good frame~\eqref{eq:DecompositionAngular} with the unit-scale decomposition of frequency space from~\eqref{eq:UnitScaleDecFreqSpace}, we arrive at
 	\begin{equation}
 	\label{eq:decf}
 		P_M f = \sum_{j \in \Z^4} \sum_{k = 0}^\infty \sum_{l = 1}^{N_k} a_k M^{-1} P_j \Big[ r^{-1} b_{k,l}(\theta) \int_0^\infty \hat{c}_{k,l}^M(\rho) J_{k+1}(M r \rho) \rho^2 \dd \rho \Big].
 	\end{equation}
	For every $M \in 2^{\Z}$ let $(X^M_{j,k,l})_{j \in \Z^4, l \in \{1, \ldots, N_k\}, k \in \N_0}$ be a sequence of independent, real-valued, mean-zero random variables on a probability space $(\Omega, \cA, \PP)$ such that there is a constant $c > 0$ with
	\begin{equation}
	\label{eq:ConditionRandomVariables}
		\int_\R e^{\gamma x} \dd \mu_{j,k,l}^M(x) \leq e^{c \gamma^2}
	\end{equation}
	for all $\gamma \in \R$, $l \in \{1, \ldots, N_k\}$, $k \in \N_0$, and $j \in \Z^4$, where $\mu^M_{j,k,l}$ denotes the corresponding distributions. A sequence of independent, mean-zero Gaussian random variables with uniformly bounded variances satisfies these assumptions, another class of examples is given by random variables whose distributions have compact support.
	
	 In view of~\eqref{eq:decf}, we then define the \emph{randomization} $f^\om$ of $f$ as
	\begin{equation}
		\label{eq:DefRandomization1}
		f^\om = \sum_{M \in 2^{\Z}} \sum_{j \in \Z^4} \sum_{k = 0}^\infty \sum_{l = 1}^{N_k} X^M_{j,k,l}(\om) a_k M^{-1} P_j \Big[ r^{-1} b_{k,l}(\theta) \int_0^\infty \hat{c}_{k,l}^M(\rho) J_{k+1}(M r \rho) \rho^2 \dd \rho \Big],
	\end{equation}
	which is understood as the limit in $L^2(\Omega, L^2(\R^4))$. As we are interested in the scaling-supercritical regime of~\eqref{eq:CubicNLSDef}, it is important to note that this randomization does not lead to higher Sobolev regularity in general. In fact, if $(\Omega, \cA, \PP)$ is the product space of two probability spaces $(\Omega_1, \cA_1, \PP_1)$ and $(\Omega_2, \cA_2, \PP_2)$ and $X^{M}_{j,k,l}(\omega) = X_j(\omega_1) X^{M}_{k,l}(\omega_2)$ for all $\omega = (\omega_1, \omega_2) \in \Omega$, where $(X_j)_{j}$ and $(X^M_{k,l})_{M,k,l}$ are independent and identically distributed Gaussian random variables on $\Omega_1$ and $\Omega_2$ respectively,
	then $f \notin H^s(\R^4)$ implies $f^\om \notin H^s(\R^4)$ almost surely for every $s > 0$. To show this statement, one first uses the arguments from the proof of~\cite[Lemma~B.1]{BT2008I} in $\omega_2$ and then employs an adaption of this proof in $\omega_1$.

	We next combine the randomization with respect to a unit-scale decomposition in frequency space and a decomposition with respect to a good frame with another unit-scale decomposition in physical space. To that purpose, we apply the decomposition~\eqref{eq:decf} to $\ph_i f$ for every $i \in \Z^4$ and then sum over $i$. To fix the notation for later use, we repeat the details of the construction.
	
	 Take $f \in L^2(\R^4)$. For each $i \in \Z^4$ and $M \in 2^{\Z}$ we first rescale $P_M(\ph_i f)$ to unit frequency, i.e. we set
	\begin{align*}
		g^M_i = (P_M(\ph_i f))(M^{-1} \cdot).
	\end{align*}
	Using polar coordinates, we develop $\hat{g}^M_i$ in the good frame as
	\begin{align*}
		\hat{g}^M_i(\rho \theta) = \sum_{k = 0}^\infty \sum_{l = 1}^{N_k} \hat{c}^{M,i}_{k,l}(\rho) b_{k,l}(\theta)
	\end{align*}
	where each $\hat{c}^{M,i}_{k,l}$ is supported in $(\frac{1}{2},2)$. Applying Theorem~3.10 from~\cite{SW71} again, we obtain the representation
	\begin{align*}
		g^M_i(r \theta) = \sum_{k = 0}^\infty \sum_{l = 1}^{N_k} a_k r^{-1} b_{k,l}(\theta) \int_0^\infty \hat{c}^{M,i}_{k,l}(\rho) J_{k+1}(r \rho) \rho^2 \dd \rho
	\end{align*}
	with
	\begin{align*}
		\|g^M_i\|_{L^2_x}^2 \sim \sum_{k = 0}^\infty \sum_{l = 1}^{N_k} \|\hat{c}^{M,i}_{k,l}\|_{L^2(\rho^3 \dd \rho)}^2.
	\end{align*}
	Rescaling to frequency $M$, recalling that $(\ph_i)_{i \in \Z^4}$ is a partition of unity, and employing the unit-scale decomposition in frequency space, we obtain in analogy to~\eqref{eq:decf}
	\begin{align}
	P_M (\ph_i f) &= \sum_{j \in \Z^4} \sum_{k = 0}^\infty \sum_{l = 1}^{N_k} a_k M^{-1} P_j \Big[ r^{-1} b_{k,l}(\theta) \int_0^\infty \hat{c}_{k,l}^{M,i}(\rho) J_{k+1}(M r \rho) \rho^2 \dd \rho \Big], \label{eq:decphif} \\	
		P_M f &= \sum_{i,j \in \Z^4} \sum_{k = 0}^\infty \sum_{l = 1}^{N_k} a_k M^{-1} P_j \Big[ r^{-1} b_{k,l}(\theta) \int_0^\infty \hat{c}_{k,l}^{M,i}(\rho) J_{k+1}(M r \rho) \rho^2 \dd \rho \Big]. \label{eq:decf2}
	\end{align}
	For every $M \in 2^\Z$ we take a sequence of independent, real-valued, mean-zero random variables $(X^M_{i,j,k,l})_{i,j \in \Z^4,l \in \{1, \ldots, N_k\}, k \in \N_0}$ on a probability space $(\Omega,\cA,\PP)$ such that there exists a constant $c > 0$ with
	\begin{align*}
		\int_{\R} e^{\gamma x} \dd \mu^M_{i,j,k,l}(x) \leq e^{c\gamma^2}
	\end{align*}
	for all $\gamma \in \R$, $l \in \{1, \ldots, N_k\}$, $k \in \N_0$, and $i,j \in \Z^4$, where $\mu^M_{i,j,k,l}$ is the distribution of $X^M_{i,j,k,l}$. Considering~\eqref{eq:decf2}, we define the second \emph{randomization} of $f$ as
	\begin{align}
		\label{eq:RandomizationFreqAngleSpace}
		f^{\tilde{\omega}} &= \sum_{M \in 2^{\Z}} \sum_{i,j \in \Z^4} \sum_{k = 0}^\infty \sum_{l = 1}^{N_k} X^{M}_{i,j,k,l}(\tilde{\omega}) a_k M^{-1} \nonumber \\
		&\hspace{10em} \cdot P_j \Big[ r^{-1} b_{k,l}(\theta) \int_0^\infty \hat{c}_{k,l}^{M,i}(\rho) J_{k+1}(M r \rho) \rho^2 \dd \rho \Big].
	\end{align}

	Again, this randomization does not imply higher Sobolev regularity in general: Assume that $(\Omega, \cA, \PP)$ is the product space of $(\Omega_i, \cA_i, \PP_i)$, $i = 1,2,3$, and that $X^{M}_{i,j,k,l}(\tilde{\omega}) = X_j(\omega_1) X^{M}_{k,l}(\omega_2) X_i(\omega_3)$ for all $\tilde{\omega} = (\omega_1, \omega_2, \omega_3) \in \Omega$, where $(X_j)_j$, $(X^{M}_{k,l})_{M,k,l}$, and $(X_i)_i$ are independent and identically distributed Gaussian random variables on $\Omega_1$, $\Omega_2$, and $\Omega_3$, respectively. One can then show that for $s > 0$, $\sum_{i\in \Z^4} X_i(\omega_3) \ph_i f \in H^s(\R^4)$ for all $\omega_3$ from a subset of $\Omega_3$ with positive measure implies that $\ph_i f$ belongs to $H^s(\R^4)$ for all $i \in \Z^4$. Applying adaptions of the proof of~\cite[Lemma~B.1]{BT2008I} successively in $\omega_3$, $\omega_2$, and $\omega_1$, we obtain that if $f \notin H^s(\R^4)$, then $f^{\tilde{\omega}} \notin H^s(\R^4)$ almost surely for every $s > 0$.
	\subsection{Main results}
	The first main result states that the energy-critical cubic NLS in four dimensions is almost surely locally wellposed for scaling-supercritical initial data from $H^s(\R^4)$ with $\frac{1}{7} < s < 1$. As usual, the sign in front of the nonlinearity does not affect the local theory and the result applies to both the defocusing and the focusing equation. We refer to~\eqref{eq:DefX} below for the definition of the function space $X(I)$ appearing in the theorem, which is a subspace of $L^2(I, \dot{B}^1_{4,2}(\R^4))$.
	\begin{thm}
	\label{thm:locwp}
	Let $\frac{1}{7} < s < 1$. Take $f \in H^s(\R^4)$ and let $f^\om$ denote the randomization of $f$ defined in~\eqref{eq:DefRandomization1}. Then for almost every $\om \in \Omega$ there is an open interval $I$ containing $0$ and a unique solution
	\begin{align*}
		u \in e^{i t \Delta} f^\om + C(I, \dot{H}^1(\R^4)) \cap X(I)
	\end{align*}
of the cubic nonlinear Schr{\"o}dinger equation
	\begin{equation}
		\label{eq:nlSchrdfandf}
		\begin{aligned}
			\imu \partial_t u + \Delta u &= \pm |u|^2 u, \\
			u(0) &= f^\om.
		\end{aligned}
	\end{equation}
\end{thm}
For the proof we set $u(t) = e^{i t \Delta} f^\om + v(t)$ and study the forced cubic NLS
\begin{equation}
		\label{eq:ForcednlSchrdfandf}
		\begin{aligned}
			\imu \partial_t v + \Delta v &= \pm |F+v|^2 (F+v), \\
			v(0) &= v_0
		\end{aligned}
	\end{equation}
for the nonlinear part $v$. We develop a suitable functional framework in which we prove local wellposedness for~\eqref{eq:ForcednlSchrdfandf}, which then implies Theorem~\ref{thm:locwp} if the linear evolution $e^{\imu t \Delta} f^\om$ almost surely fulfills the assumptions on the forcing $F$. For the latter, it is crucial to exploit that the randomization improves the space-time integrability of $e^{\imu t \Delta} f^\om$ in various ways.

For example, the Wiener randomization implies a unit-scale Bernstein inequality which allows to move from high spatial integrability to  lower spatial integrability without losing derivatives. Combined with the standard deterministic Strichartz estimates, we obtain almost surely much better space-time integrability of the linear flow. This improvement was combined with the bilinear refinement of Strichartz estimates in~\cite{BOP2015I, BOP2015II} to prove almost sure local wellposedness of~\eqref{eq:nlSchrdfandf} with initial data from $H^s(\R^4)$ with $\frac{3}{5} < s < 1$. The regularity threshold was lowered to $\frac{1}{3} < s < 1$ in~\cite{DLM2019} by the use of local smoothing estimates.

In this work, we make use of a new mechanism to gain regularity for the nonlinear part of the equation. We show that, as in the case of the wave equation~\cite{BK2019, S2020}, the angular randomization gives access to Strichartz estimates with (almost) the same range of exponents as for the radial Schr{\"o}dinger equation. In the Schr{\"o}dinger case, however, we can choose space-time norms which gain derivatives by scaling. In fact, we can gain $\frac{3}{7}-$ derivatives in that way. Using the optimal norms both for the homogeneous and the inhomogeneous estimate, we increase the regularity by $\frac{6}{7}-$ derivatives. In contrast to~\cite{DLM2019}, where maximal function estimates had to be employed, the combination of the Wiener and the angular randomization is so flexible that we do not lose any derivatives when closing the estimates. Although the local smoothing norms lead to a higher regularity gain, we thus obtain the lower threshold of $\frac{1}{7}$. Consequently, a key observation for the proof is the probabilistic estimate for $e^{\imu t \Delta} f^\om$ with the randomization~\eqref{eq:DefRandomization1} in Proposition~\ref{prop:ProbabilisticEstimate}, where we show that we can exploit the advantages of the Wiener and the angular randomization simultaneously.

We next come to our second main result which establishes almost sure scattering for solutions of~\eqref{eq:CubicNLSDef} with randomized initial data from $H^s(\R^4)$ with $\frac{5}{7} < s < 1$, where the randomization from~\eqref{eq:RandomizationFreqAngleSpace} is applied. In particular, we do not assume that the initial data before the randomization is radially symmetric.
\begin{thm}
	\label{thm:AlmostSureScattering}
	Let $\frac{5}{7} < s < 1$. Take $f \in H^s(\R^4)$ and let $f^{\tilde{\omega}}$ denote the randomization of $f$ from~\eqref{eq:RandomizationFreqAngleSpace}. Then for almost all $\tilde{\omega} \in \Omega$ there is a unique global solution 
	\begin{align*}
		u \in e^{\imu t \Delta} f^{\tilde{\omega}} + C(\R, \dot{H}^1(\R^4))\cap X(\R)
	\end{align*}
	which scatters both forward and backward in time, i.e., there exist $v_{\pm} \in \dot{H}^1(\R^4)$ such that 
	\begin{align*}
		\lim_{t \rightarrow \pm \infty} \|u(t) - e^{\imu t \Delta}(f^{\tilde{\omega}} + v_{\pm})\|_{\dot{H}^1} = 0.
	\end{align*}
\end{thm}
For the proof of this result, we follow the same strategy as~\cite{DLM2020, KMV2019, DLM2019, Br2019}. Starting from the local wellposedness theory of~\eqref{eq:ForcednlSchrdfandf}, we develop a perturbation theory in the corresponding functional framework, which yields a conditional scattering result: If the energy of $v$, i.e. $E(v)$, is bounded on its maximal interval of existence, the solution exists globally and scatters both forward and backward in time. See Proposition~\ref{prop:ConditionalScattering} for the precise formulation. The proof of Theorem~\ref{thm:AlmostSureScattering} thus reduces to the task of controlling the energy of $v$. We point out that the energy of $v$ is not conserved as $v$ is not a solution of the cubic NLS.

To control the energy of $v$, we then estimate the energy increment $\partial_t E(v(t))$. The most difficult term we get with this approach is $\| (\nabla F) (\nabla v) v^2 \|_{L^1_t L^1_x}$. If one only has the energy to estimate the factors involving $v$, one is led to control $\nabla F$ in $L^1_t L^\infty_x$.

In~\cite{KMV2019} and~\cite{DLM2019} this problem is overcome by not only invoking the energy to control $v$ but also a Morawetz type inequality to avoid this $L^1_t L^\infty_x$-estimate. Both the energy and the Morawetz term are then controlled by a double bootstrap argument. However, this approach led to place $\nabla F$ in spatially weighted spaces which in turn required to randomize \emph{radially symmetric} data so that the linear flow satisfied the assumptions on $F$ almost surely. This led to almost sure scattering for solutions of~\eqref{eq:CubicNLSDef} with Wiener-randomized radially symmetric initial data from $H^s(\R^4)$ with $\frac{5}{6} < s < 1$ in~\cite{KMV2019} and $\frac{1}{2} < s < 1$ in~\cite{DLM2019}.

In order to get rid of the radial symmetry assumption, we indeed only use the energy to estimate the factors involving $v$ in $\| (\nabla F) (\nabla v) v^2 \|_{L^1_t L^1_x}$, assuming that $\nabla F$ belongs to $L^1_t L^\infty_x$. Hence, the key ingredient in the proof of Theorem~\ref{thm:AlmostSureScattering} is to show that $\nabla e^{\imu t \Delta} f^{\tilde{\omega}}$ belongs to $L^1_t L^\infty_x$ almost surely. In fact, this is the only condition on $F$ for which we need the randomization~\eqref{eq:RandomizationFreqAngleSpace} in order that the linear flow of the randomized data satisfies it almost surely. For all other assumptions the randomization~\eqref{eq:DefRandomization1} would suffice.

To estimate $\nabla e^{\imu t \Delta} f^{\tilde{\omega}}$ in $L^1_t L^\infty_x$, we exploit that the randomization~\eqref{eq:RandomizationFreqAngleSpace} yields the same improved decay of the linear flow as the pure physical-space randomization in~\cite{NY2019,S2020} (see Proposition~\ref{prop:ProbabilisticEstimate2}). On the other hand, we can still gain regularity for the linear flow as for the randomization~\eqref{eq:DefRandomization1} (see Proposition~\ref{prop:ProbabilisticEstimate}). However, we cannot exploit the different advantages simultaneously but we have to interpolate between the different effects, which leads to the regularity restriction $s > \frac{5}{7}$.

\begin{rem}
	\label{rem:ComparisonToWave}
	In~\cite{Br2019}, almost sure scattering for the cubic nonlinear wave equation in $\R^4$ without radial symmetry was shown using a randomization based on a unit-scale decomposition in physical and frequency space, i.e.~\eqref{eq:RandomizationFreqAngleSpace} without the randomization in the angular variable. In contrast to the nonlinear wave equation, the energy of~\eqref{eq:CubicNLSDef} does not control the time derivative of the solution and the (almost sure) regularizing effect of the linear Schr{\"o}dinger flow we obtain from the randomization in the angular variable is an important component in the proof of the almost sure scattering result. On the other hand, using randomization~\eqref{eq:RandomizationFreqAngleSpace} for the cubic defocusing nonlinear wave equation in $\R^4$, one can lower the regularity threshold for almost sure scattering from~\cite{Br2019} with a much simpler proof, see~\cite{Swave2021}.
\end{rem}

The rest of the paper is organized as follows. In Section~\ref{sec:NotationPrelim} we fix some notation, introduce the functional framework for the forced equation~\eqref{eq:ForcednlSchrdfandf}, and collect several deterministic estimates which will be used in the following. The crucial probabilistic estimates for the linear Schr{\"o}dinger flow with data randomized with the techniques introduced in this work are presented in Section~\ref{sec:ProbabilisticEst}. The almost sure local wellposedness theory is developed in Section~\ref{sec:AlmostSureLocWP}. The perturbation theory and the conditional scattering result are derived in Section~\ref{sec:CondScattering}. In Section~\ref{sec:AlmostSureScattering} we then prove the almost sure scattering result.
\section{Notation and preliminaries}
\label{sec:NotationPrelim}

In this section we fix some notation, collect several deterministic estimates, and introduce the functional framework for the local wellposedness theory of the forced cubic NLS~\eqref{eq:ForcednlSchrdfandf}.

\subsection{Notation}
\label{subsec:Notation}

We write $A \lesssim B$ if there is a constant $C > 0$ such that $A \leq C B$ and $A \sim B$ if $A \lesssim B$ and $B \lesssim A$.

Fix an even function $\eta_0 \in C_c^\infty(\R)$ such that $0 \leq \eta_0 \leq 1$, $\eta_0(x) = 1$ for $|x| \leq \frac{5}{4}$, and $\eta_0(x) = 0$ for $|x| \geq \frac{8}{5}$. For every dyadic number $N \in 2^\Z$ we define the symbols
\begin{align*}
	\chi_N(\xi) = \eta_0(|\xi|/N) - \eta_0(2|\xi|/N), \qquad \chi_{\leq N}(\xi) = \eta_0(|\xi|/N)
\end{align*}
as well as the standard Littlewood-Paley projectors
\begin{align*}
	P_N f = \cF^{-1}(\chi_N \hat{f}), \qquad P_{\leq N} f = \cF^{-1}(\chi_{\leq N} \hat{f}),
\end{align*}
where $\hat{f} = \cF f$ denotes the Fourier transform of $f$. We always use upper case letters (which represent a dyadic number) to indicate a Littlewood-Paley projector, while a lower case letter (representing an element of $\Z^4$) refers to the operators introduced in~\eqref{eq:DefPk}. We further define the fattened Littlewood-Paley projectors
\begin{align*}
	\tilde{P}_N = \sum_{|\log_2 (M/N)| \leq 4}P_{M}.
\end{align*}

Let $p,\mu \in [1,\infty]$ and set $\cL^p(0,\infty) = L^p((0,\infty), r^3 \dd r)$. We will employ the spaces $\cL^p L^\mu(\R^4)$, anisotropic in the radial and the angular variable, which are defined by the norms
\begin{align*}
	\|f\|_{\cL^p_r L^\mu_\theta} = \Big(\int_0^\infty \Big(\int_{S^3} |f(r\theta)|^\mu \dd \theta \Big)^{\frac{p}{\mu}} r^3 \dd r \Big)^{\frac{1}{p}}
\end{align*}
with the usual adaptions in the case $p = \infty$ or $\mu = \infty$. 

Let $I \subseteq \R$ be an interval. If $X$ is a function space over $\R^4$ with norm $\|\cdot\|_X$, we write $\| \cdot \|_{L^q_t X}$ for the norm
\begin{align*}
	\|f\|_{L^q_t X} = \Big( \int_I \|f(t)\|_X^q \dd t \Big)^{\frac{1}{q}}
\end{align*}
of $L^q(I, X)$ without specifying the underlying interval $I$ if it is clear from the context. If we want to highlight the time interval, we write $\|f\|_{L^q_I X}$.

We use the standard homogeneous Besov space $\dot{B}^s_{p,2}(\R^4)$ as well as the Besov-type spaces $\dot{B}^s_{(p,\mu),2}(\R^4)$ and $\dot{B}^s_{q,(p,\mu),2}(I \times \R^4)$ defined by the norms
\begin{align*}
	\|f\|_{\dot{B}^s_{p,2}} &= \Big(\sum_{N \in 2^\Z} N^{2s} \|P_N f\|_{L^p}^2 \Big)^{\frac{1}{2}}, \qquad
	\|f\|_{\dot{B}^s_{(p,\mu),2}} = \Big(\sum_{N \in 2^\Z} N^{2 s} \|P_N f\|_{\cL^p_r L^\mu_\theta}^2 \Big)^{\frac{1}{2}}, \\
	\|g\|_{\dot{B}^s_{q,(p,\mu),2}} &= \Big(\sum_{N \in 2^\Z} N^{2 s} \|P_N g\|_{L^q_t \cL^p_r L^\mu_\theta}^2 \Big)^{\frac{1}{2}}
\end{align*}
for $s \in \R$. Finally, we set $\dot{B}^s_{q,p,2}(I \times \R^4) = \dot{B}^s_{q,(p,p),2}(I \times \R^4)$ in the case $\mu = p$.

\subsection{Deterministic estimates} We start by collecting the Strichartz estimates we will employ later. To that purpose we recall that in dimension $d = 4$ a pair $(q,p)$ is \emph{Schr{\"o}dinger admissible} if
\begin{align*}
	2 \leq q,p \leq \infty, \qquad \frac{2}{q} + \frac{4}{p} = 2.
\end{align*}
We then have the classical Strichartz estimates, see~\cite{KT1998}, where we denote the dual exponent of an integrability exponent $p$ by $p'$.
\begin{lem}
	\label{lem:StrichartzEstimates}
	Let $I \subseteq \R$ be an interval and $t_0 \in I$.
	Let $(q,p)$, $(\tilde{q}, \tilde{p})$ be Schr{\"o}dinger admissible. Then
		\begin{align*}
		\| e^{\imu t \Delta} f\|_{L^q_t L^p_x} \lesssim \|f\|_{L^2}, \qquad
		\Big\| \int_{t_0}^t e^{\imu (t-s) \Delta} g(s) \dd s \Big\|_{L^q_t L^p_x} \lesssim \|g\|_{L^{\tilde{q}'}_t L^{\tilde{p}'}_x}.
		\end{align*}
\end{lem}
Besides the standard Strichartz estimates, we also use the ones in spherically averaged spaces. Therefore, we also consider the extended range
\begin{equation}
\label{eq:ExtendedRange}
	2 \leq q,p \leq \infty, \qquad \frac{1}{q} \leq \frac{7}{2}\Big(\frac{1}{2} - \frac{1}{p}\Big), \qquad (q,p) \neq \Big(2, \frac{14}{5} \Big)
\end{equation}
for exponent pairs $(q,p)$. We obtain the following estimate. 
\begin{lem}
	\label{lem:StrichartzEstimatesAv}
	Let $I \subseteq \R$ be an interval and $t_0 \in I$.
 Let $(q,p)$ be Schr{\"o}dinger admissible and $(\tilde{q}, \tilde{p})$ satisfy~\eqref{eq:ExtendedRange} with $q > \tilde{q}'$. Then
		\begin{align*}
			\Big\|\int_{t_0}^t e^{\imu (t-s) \Delta} P_N g(s) \dd s \Big\|_{L^q_t L^p_x} \lesssim N^{2-\frac{2}{\tilde{q}}-\frac{4}{\tilde{p}}}\|P_N g\|_{L^{\tilde{q}'}_t \cL^{\tilde{p}'}_r L^2_\theta}
		\end{align*}
		for all $N \in 2^\Z$.
\end{lem}

\begin{proof}
	As $(\tilde{q},\tilde{p})$ satisfies~\eqref{eq:ExtendedRange}, Theorem~1.1 in~\cite{Guo2016} yields
	\begin{equation}
	\label{eq:StrichartzAvHom}
		\|e^{\imu t \Delta} P_0 f\|_{L^{\tilde{q}}_t \cL^{\tilde{p}}_r L^2_\theta} \lesssim \|P_0 f\|_{L^2}
	\end{equation}
	for every $f \in L^2(\R^4)$. Applying Lemma~\ref{lem:StrichartzEstimates} first and then the dual estimate of~\eqref{eq:StrichartzAvHom}, we infer
	\begin{align*}
		\Big\|\int_I e^{\imu (t-s) \Delta} P_0 g(s) \dd s \Big\|_{L^q_t L^p_x} \lesssim \Big\|\int_I e^{-\imu s \Delta} P_0 g(s) \dd s \Big\|_{L^2} \lesssim \|P_0 g\|_{L^{\tilde{q}'}_t \cL^{\tilde{p}'}_r L^2_\theta}.
	\end{align*}
	Employing the Christ-Kiselev lemma (see~\cite{CK2001,T2000}) and rescaling to frequency $N$, the assertion follows.
\end{proof}

We next turn to a square function estimate which will be crucial in the proof of our probabilistic estimates in Section~\ref{sec:ProbabilisticEst}.
\begin{lem}
	\label{lem:SquareFunction}
	Let $p,\mu \in [2,\infty)$. Then for every $\tilde{p} \in [2,p]$ we have
	\begin{align*}
		\Big\|\Big(\sum_{j \in \Z^4} |P_j f|^2 \Big)^{\frac{1}{2}} \Big\|_{\cL^p_r L^\mu_\theta} \lesssim \|f\|_{\cL^{\tilde{p}}_r L^\mu_\theta}.
	\end{align*}
\end{lem}

\begin{proof}
We first assume that $f$ is a Schwartz function. In the proof of Lemma~2.8 in~\cite{KMV2019} it was shown that
	\begin{equation}
		\label{eq:SquareFunctionEstimate}
		\Big(\sum_{j \in \Z^4} |P_j f|^2 \Big)^{\frac{1}{2}} \lesssim (|\check{\psi}| \ast |f|^2)^{\frac{1}{2}},
	\end{equation}
	where $\check{\psi}$ is the inverse Fourier transform of the function $\psi$ from Subsection~\ref{subsec:UnitScaleDecompositionFreq}. We define $\hat{p} \in [1,\infty)$ by
	\begin{align*}
		\frac{1}{\hat{p}} = 1 - 2\Big(\frac{1}{\tilde{p}} - \frac{1}{p} \Big).
	\end{align*}
	Applying Young's convolution estimate for spherically averaged spaces, see Lemma~3.2 in~\cite{GLNW2014}, we obtain
	\begin{align*}
		\Big\|\Big(\sum_{j \in \Z^4} |P_j f|^2 \Big)^{\frac{1}{2}} \Big\|_{\cL^p_r L^\mu_\theta} \lesssim \||\check{\psi}| \ast |f|^2\|_{\cL^{\frac{p}{2}}_r L^{\frac{\mu}{2}}_\theta}^{\frac{1}{2}} \lesssim \| \check{\psi} \|_{\cL^{\hat{p}}_r L^\infty_\theta}^{\frac{1}{2}} \| |f|^2 \|_{\cL^{\frac{\tilde{p}}{2}}_r L^{\frac{\mu}{2}}_\theta}^{\frac{1}{2}} \lesssim \|f\|_{\cL^{\tilde{p}}_r L^\mu_\theta},
	\end{align*}
	where we used that
	\begin{align*}
		\sup_{\theta \in S^3} |\check{\psi}(r \theta)| = \sup_{|x| = r} |\check{\psi}(x)| \leq (1 + r^2)^{-3} \sup_{|x| = r} (1 + |x|^2)^3 |\check{\psi}(x)| \lesssim (1 + r^2)^{-3}
	\end{align*}
	for all $r > 0$ since $\check{\psi}$ is a Schwartz function.
	
	The assertion of the lemma then follows by approximation.
\end{proof}

\subsection{Framework for the forced cubic nonlinear Schr{\"o}dinger equation}
\label{subsec:FunctionalFramework}
In this subsection we introduce the functional setup in which we develop the local wellposedness and the perturbation theory for the forced cubic NLS~\eqref{eq:ForcednlSchrdfandf}. The involved norms depend on a fixed small parameter $0 < \delta \ll 1$. In the proof of Theorem~\ref{thm:locwp} it will be chosen in such a way that $\frac{1}{7} + 6 \delta \leq s$, where $\frac{1}{7} < s < 1$ is the regularity parameter from that theorem. We further set $\nu = \frac{16}{7} \delta$.

The norms are defined dyadically. Let $I \subseteq \R$ be an interval. For $v_N$ of spatial frequency $N$, i.e. having spatial Fourier support in $A_N := I \times \{\xi \in \R^4 \colon N/2 < |\xi| < 2N\}$, we set
\begin{equation}
	\label{eq:DefXN}
	\|v_N\|_{X_N(I)} = N\|v_N\|_{L^{\frac{1}{\delta}}_t L^{\frac{2}{1-\delta}}_x} + N \|v_N\|_{L^2_t L^4_x}
\end{equation}
for all $N \in 2^{\Z}$. We then define
\begin{equation}
	\label{eq:DefX}
	\|v\|_{X(I)} = \Big(\sum_{N \in 2^{\Z}} \|P_N v\|_{X_N(I)}^2 \Big)^{\frac{1}{2}}.
\end{equation}
For the right-hand sides we introduce the norms
\begin{align*}
	\|h_N\|_{G_N(I)} = \inf\big(N\|h_N^{(1)}\|_{L^1_t L^2_x} + N^{\frac{4}{7}+\nu} \|h_N^{(2)}\|_{L^{\frac{2}{1+2\delta}}_t \cL^{\frac{14}{9+\delta}}_r L^2_\theta}\big)
\end{align*}
for all $N \in 2^{\Z}$, where $\supp \hat{h}_N \subseteq A_N$ and the infimum is taken over all $h_N^{(1)}$ and $h_N^{(2)}$ with $h_N = h_N^{(1)} + h_N^{(2)}$ and $\supp \hat{h}_N^{(1)}, \supp\hat{h}_N^{(2)} \subseteq A_N$. We set
\begin{align*}
	\|h\|_{G(I)} = \Big( \sum_{N \in 2^{\Z}} \|P_N h\|_{G_N(I)}^2 \Big)^{\frac{1}{2}}.
\end{align*}
Note that we obtain
\begin{equation}
	\label{eq:StrichartzinXLin}
	\| e^{\imu t \Delta} v_0\|_{L^\infty_I \dot{H}^1} + \|e^{\imu t \Delta} v_0 \|_{X(I)} \lesssim \|v_0\|_{\dot{H}^1}
\end{equation}
from Lemma~\ref{lem:StrichartzEstimates} as $(\frac{1}{\delta}, \frac{2}{1-\delta})$ and $(2,4)$ are Schr{\"o}dinger admissible.
Moreover, since $(\frac{2}{1-2\delta}, \frac{14}{5-\delta})$ satisfies~\eqref{eq:ExtendedRange} and
\begin{align*}
	2 - \frac{2}{\frac{2}{1-2\delta}} - \frac{4}{\frac{14}{5-\delta}} = -\frac{3}{7} + \frac{16}{7} \delta = -\frac{3}{7} + \nu,
\end{align*}
Lemma~\ref{lem:StrichartzEstimates} and Lemma~\ref{lem:StrichartzEstimatesAv} imply
\begin{equation}
	\label{eq:StrichartzinXNonlin}
	\Big\|\int_{t_0}^t e^{\imu (t-s)\Delta} h(s) \dd s \Big\|_{L^\infty_I \dot{H}^1} + \Big\|\int_{t_0}^t e^{\imu (t-s)\Delta} h(s) \dd s \Big\|_{X(I)} \lesssim \|h\|_{G(I)}.
\end{equation}

Finally, we use the norms
\begin{align}
	\|F_N\|_{Y_N(I)} &= N^{\frac{4}{7}+\nu} \langle N \rangle^{ 4 \delta}  \| F_N\|_{L^2_t \cL^{\frac{14}{5-\delta}}_r L^{\frac{4}{\delta}}_\theta} + N^{\frac{4}{7}+\nu} \langle N \rangle^{4\delta} \|F_N\|_{L^2_t \cL^{28}_r L^{\frac{4}{\delta}}_\theta} \nonumber\\
	&\qquad +\langle N \rangle^{\frac{1}{7}} \|F_N\|_{L^{\frac{1}{\delta}}_t L^{\frac{14}{5-\delta}}_x} + \langle N \rangle^{\frac{1}{7}} \|F_N\|_{L^{\frac{1}{\delta}}_t L^{28}_x} \label{eq:DefYNnorm}
\end{align}
for $F_N$ with $\supp \hat{F}_N \in A_N$ and
\begin{align*}
	\|F\|_{Y(I)} = \Big( \sum_{N \in 2^{\Z}} \|P_N F\|_{Y_N(I)}^2 \Big)^{\frac{1}{2}}
\end{align*}
for the forcing term.
\section{Probabilistic estimates}
\label{sec:ProbabilisticEst}

In this section we derive the improved space-time estimates for the linear Schr{\"o}dinger flow with data randomized with respect to the decompositions in frequency space and the angular variable or in frequency space, the angular variable, and physical space.

A crucial tool is the following large deviation estimate, which we formulate as in Lemma~3.1 in~\cite{BT2008I}.

\begin{lem}
	\label{lem:LargeDeviation}
	Let $(X_k)_{k \in \N}$ be a sequence of independent, real-valued, mean-zero random variables on a probability space $(\Omega, \cA, \PP)$ with distributions $(\mu_k)_k$ such that there is a constant $c > 0$ with
	\begin{align*}
		\int_{\R} e^{\gamma x} \dd \mu_k(x) \leq e^{c \gamma^2}
	\end{align*}
	for all $\gamma \in \R$ and $k \in \N$. Then there is a constant $C > 0$ such that
	\begin{align*}
		\Big\|\sum_{k \in \N} c_k X_k \Big\|_{L^\beta(\Omega)} \leq C \sqrt{\beta} \Big(\sum_{k \in \N} |c_k|^2 \Big)^{\frac{1}{2}}
	\end{align*}
	for all $(c_k)_{k \in \N} \in l^2(\N)$ and $\beta \in [2,\infty)$.
\end{lem}

In the analysis of the randomization~\eqref{eq:RandomizationFreqAngleSpace}, we will also need Lemma~3.3 from~\cite{S2020} since the Littlewood-Paley operators do not commute with the physical-space randomization. In~\cite{S2020}, it was formulated in three dimensions, but neither the statement nor the proof depend on the dimension. Similar estimates have previously appeared in~\cite{Br2019, DLM2019}.
\begin{lem}
	\label{lem:Mismatch}
	Let $1 \leq p < \infty$, $M,N \in 2^{\N}$ with $|\log_2 \frac{M}{N}| \geq 5$, $l,l' \in \Z^4$, $D > 0$, and $(\ph_l)_{l \in \Z^4}$ the partition of unity from~\eqref{eq:DefPhysicalSpacePartition}. Then
	\begin{align}
		&\|\ph_l P_M (\ph_{l'} f)\|_{L^p} + \|\ph_l P_{\leq 2^0} (\ph_{l'}f)\|_{L^p} \lesssim_D \langle l - l' \rangle^{-D} \|f\|_{L^p}, \label{eq:MismatchSpace} \\
		&\|P_M(\ph_l P_N f)\|_{L^p}  \lesssim_D M^{-D} N^{-D} \|f\|_{L^p}, \label{eq:MismatchFrequencyHH} \\
		&\|P_M(\ph_l P_{\leq 2^{-5}M} f)\|_{L^p} \lesssim_D M^{-D} \|f\|_{L^p}, \label{eq:MismatchFrequencyHL}
	\end{align}
	for all $f \in L^p(\R^4)$, where the implicit constants are independent of $M,N,l$ and $l'$.
\end{lem}

The estimates from the above lemma imply the following corollary, which is suitable for our applications. It is analogous to Corollary~3.4 in~\cite{S2020} where the case $s = 1$ was treated but we provide the details of the proof as we consider general regularities $s$ here.
\begin{cor}
	\label{cor:CompLPPhysicalSpace}
	Let $p \in [2,\infty)$, $s \geq 0$, and $(\ph_l)_{l \in \Z^4}$ be the partition of unity introduced in~\eqref{eq:DefPhysicalSpacePartition}. We then have
	\begin{align*}
		\| \langle M \rangle^s P_M(\ph_l f)\|_{l^2_M l^2_l L^{p'}_x} \lesssim \|f\|_{H^s}
	\end{align*}
	for all $f \in H^s(\R^4)$.
\end{cor}

\begin{proof}
	We first consider $M \in 2^{\Z}$ with $M \leq 2^4$. Using Bernstein's and H{\"o}lder's inequalities combined with $|\supp \ph_l| \lesssim 1$ for all $l \in \Z^4$, we infer
	\begin{align}
		\label{eq:EstFreqPhysSpLow}
		\| \langle M \rangle^s P_M (\ph_l f)\|_{l^2_l L^{p'}_x} &\lesssim \langle M \rangle^s M^{4(1-\frac{1}{p'})} \|\ph_l f\|_{l^2_l L^1_x} \lesssim \langle M \rangle^s M^{\frac{4}{p}} \| \ph_l f\|_{l^2_l L^2_x} \nonumber\\
		&\lesssim \langle M \rangle^s M^{\frac{4}{p}} \|f\|_{L^2}.
	\end{align}
	Next fix $M \in 2^{\Z}$ with $M \geq 2^5$. We observe that
	\begin{align}
		\|P_M(\ph_l f)\|_{l^2_l L^{p'}_x} &\leq \|P_M (\ph_l \tilde{P}_M f)\|_{l^2_l L^{p'}_x} + \|P_M(\ph_l P_{\leq 0} f)\|_{l^2_l L^{p'}_x} \nonumber\\
		&\hspace{3em} + \sum_{\substack{ N \in 2^\N \\ |\log_2 (M/N)| \geq 5}} \|P_M (\ph_l P_N f)\|_{l^2_l L^{p'}_x}. \label{eq:EstFreqPhysSpHigh}
	\end{align}
	For the first term on the right-hand side we estimate
	\begin{align*}
		\|P_M (\ph_l \tilde{P}_M f)\|_{l^2_l L^{p'}_x} \lesssim \|\ph_l \tilde{P}_M f\|_{l^2_l L^2_x} \lesssim \|\tilde{P}_M f\|_{L^2_x}, 
	\end{align*}
	where we again used H{\"o}lder's inequality combined with $|\supp \ph_l| \lesssim 1$ for all $l \in \Z^4$. To treat the sum on the right-hand side of~\eqref{eq:EstFreqPhysSpHigh}, we first set $\tilde{\ph}_l = \sum_{m \in \Z^4, |m-l| \leq 4} \ph_m$ so that $\tilde{\ph}_l = 1$ on the support of $\ph_l$ for every $l \in \Z^4$. Applying Lemma~\ref{lem:Mismatch}, we obtain
	\begin{align*}
		\|P_M (\ph_l P_N f)\|_{L^{p'}_x} &\leq \sum_{l' \in \Z^4} \|P_M (\ph_l P_N (\ph_{l'}f))\|_{L^{p'}_x} \\
		&\lesssim  \sum_{l' \in \Z^4} \|P_M (\ph_l P_N (\ph_{l'}f))\|_{L^{p'}_x}^{\frac{1}{2}} \|\ph_l P_N (\ph_{l'}f)\|_{L^{p'}_x}^{\frac{1}{2}} \\
		&\lesssim \sum_{l' \in \Z^4} M^{-D} N^{-D} \langle l - l' \rangle^{-20} \| \tilde{\ph}_{l'} f\|_{L^{p'}_x} 
	\end{align*}
	for all $N \in 2^\N$ with $|\log_2(M/N)| \geq 5$ and $l \in \Z^4$ with a constant $D > s$. Exploiting that $\langle x - l\rangle \lesssim \langle l - l' \rangle$ for all $x \in \supp \tilde{\ph}_{l'}$ as well as $|\supp \tilde{\ph}_{l'}| \lesssim 1$ for all $l' \in \Z^4$, we get
	\begin{align*}
		&\|P_M (\ph_l P_N f)\|_{L^{p'}_x} \lesssim \sum_{l' \in \Z^4} M^{-D} N^{-D} \langle l - l' \rangle^{-20} \| \tilde{\ph}_{l'} f\|_{L^{2}_x} \\
		&\lesssim \sum_{l' \in \Z^4} M^{-D} N^{-D} \langle l - l' \rangle^{-10} \| \langle x - l\rangle^{-10} f\|_{L^{2}_x} \lesssim M^{-D} N^{-D}  \| \langle x - l\rangle^{-10} f\|_{L^{2}_x}.
	\end{align*}
	Taking the $l^2$-norm in $l$ and summing over $N \in 2^{\N}$ with $\log_2(M/N) \geq 5$, we arrive at
	\begin{align*}
		 \sum_{\substack{ N \in 2^\N \\ |\log_2 (M/N)| \geq 5}} \|P_M (\ph_l P_N f)\|_{l^2_l L^{p'}_x} 
		  \lesssim M^{-D} \|f\|_{L^2_x}.
	\end{align*}
	Arguing in the same way for the second term on the right-hand side of~\eqref{eq:EstFreqPhysSpHigh}, we get
	\begin{align*}
		\|P_M (\ph_l P_{\leq 0} f)\|_{l^2_l L^{p'}_x} \lesssim M^{-D} \|f\|_{L^2_x}.
	\end{align*}
	Inserting these estimates in~\eqref{eq:EstFreqPhysSpHigh} and multiplying with $\langle M \rangle^s$, we conclude
	\begin{equation}
		\label{eq:EstFreqPhysSpHigh2}
		\|\langle M \rangle^s P_M(\ph_l f)\|_{l^2_l L^{p'}_x} \lesssim \langle M \rangle^s \|\tilde{P}_M f\|_{L^2_x} + \langle M \rangle^s M^{-D} \|f\|_{L^2_x}.
	\end{equation}
	Finally, we take the $l^2$-norm in~\eqref{eq:EstFreqPhysSpHigh2} over $M \in 2^{\Z}$ with $M \geq 2^5$ and the $l^2$-norm in~\eqref{eq:EstFreqPhysSpLow} over $M \in 2^{\Z}$ with $M \leq 2^4$, which yields the assertion of the corollary as $D > s$.
\end{proof}

We now turn to the improvements in space-time integrability of the linear Schr{\"o}dinger flow with randomized data. We first consider the effect of the randomization with respect to the decomposition in frequency space and the angular variable. It allows us to move from high integrability spaces to spaces of lower integrability without losing regularity. Unlike the case of the pure Wiener randomization, the lower integrability threshold is not dictated by the Schr{\"o}dinger admissibility condition, but by~\eqref{eq:ExtendedRange} as we can make use of Strichartz estimates in that range of exponents due to the additional randomization in the angular variable.

In the deterministic setting, Strichartz estimates for the range~\eqref{eq:ExtendedRange} are obtained in spherically averaged spaces, i.e. in spaces with weaker angular integrability, see~\cite{Guo2016}. Exploiting that one can take an orthonormal basis of $L^2(S^3)$ consisting of spherical harmonics, these estimates reduce to uniform estimates for a family of Bessel functions. In fact, the Strichartz estimates in spherically averaged spaces are reduced to the estimate
\begin{equation}
	\label{eq:EstimateT2kpl1}
	\|T_2^{k+1}(h)\|_{L^q_t \cL^p_r} \lesssim \|h\|_{L^2_r}
\end{equation}
uniformly in $k$, see~(2.7) in~\cite{Guo2016}, where the operators $T_2^{k+1}$ are defined by
\begin{equation}
\label{eq:DefT2kpl1}
	T_2^{k+1}(h)(t,r) = r^{-1} \int_0^\infty e^{-\imu t \rho^2}J_{k+1}(r \rho) \chi_{2^0}(\rho) h(\rho)\rho^2 \dd\rho
\end{equation}
for all $k \in \N_0$ and $h \in L^2(0,\infty)$.
Due to the randomization in the angular variable, we can make use of estimate~\eqref{eq:EstimateT2kpl1}.

In order to optimally exploit the advantages of the randomization in frequency space and the angular variable as described above, it is crucial that we do not rely on the unit-scale Bernstein inequality as it is typically done in the case of the Wiener randomization, but employ the square function estimate from Lemma~\ref{lem:SquareFunction} to move from high to low integrability spaces.

Finally, we note that the additional randomization in physical space does not undermine the above advantages so that we obtain the same set of estimates for the linear Schr{\"o}dinger flow with data randomized with respect to the decomposition in frequency space, the angular variable, and physical space.


\begin{prop}
	\label{prop:ProbabilisticEstimate}
	Let $q,p_0 \in [2,\infty)$ such that $(q,p_0)$ satisfies~\eqref{eq:ExtendedRange} and let $s \geq 0$ and $s' \in \R$. 
	\begin{enumerate}	
		\item \label{it:ProbabilisticEstimate1om} Take $f \in H^s(\R^4)$ and let $f^\om$ be its randomization from~\eqref{eq:DefRandomization1}. Then there is a constant $C > 0$ such that
	\begin{align*}
		\|\langle \nabla \rangle^{s'} e^{\imu t \Delta} f^\om \|_{L^\beta_\om \dot{B}^{s + \frac{2}{q} + \frac{4}{p_0} - 2}_{q,(p,\mu),2}} \leq C \sqrt{\beta} \|\langle \nabla \rangle^{s'} f\|_{\dot{H}^s}
	\end{align*}	 
	for all $\beta \in [1,\infty)$, $p \in [p_0,\infty)$, and $\mu \in [2,\infty)$.
	\item \label{it:ProbabilisticEstimate1tildeom} Take $f \in H^s(\R^4)$ and let $f^{\tilde{\om}}$ be its randomization from~\eqref{eq:RandomizationFreqAngleSpace}. Then there is a constant $C > 0$ such that
	\begin{align*}
		\| \langle \nabla \rangle^{s'} e^{\imu t \Delta} f^{\tilde{\om}} \|_{L^\beta_{\tilde{\om}} \dot{B}^{s + \frac{2}{q} + \frac{4}{p_0} - 2}_{q,(p,\mu),2}} \leq C \sqrt{\beta} \|f\|_{H^{s+s'}}
	\end{align*}	 
	for all $\beta \in [1,\infty)$, $p \in [p_0,\infty)$, and $\mu \in [2,\infty)$.
	\end{enumerate}
\end{prop}

\begin{proof}
	Since $\Omega$ is a probability space, it is enough to show the assertions for $\beta \geq \max\{q, p,\mu\} \geq 2$.
	
	(i) By definition and Minkowski's inequality, we have
	\begin{align*}
		&\|\langle \nabla \rangle^{s'} e^{\imu t \Delta} f^\omega \|_{L^\beta_\omega \dot{B}^{s + \frac{2}{q} + \frac{4}{p_0} - 2}_{q,(p,\mu),2}} \\
		&\sim \Big\| \Big( \sum_{M \in 2^{\Z}} \langle M \rangle^{2s'} M^{2(s + \frac{2}{q} + \frac{4}{p_0} - 2)} \| P_M e^{\imu t \Delta} f^{\omega} \|_{L^q_t \cL^p_r L^\mu_\theta}^2 \Big)^{\frac{1}{2}} \Big\|_{L^\beta_\omega } \\
		&\lesssim \Big(\sum_{M \in 2^{\Z}} \langle M \rangle^{2s'} M^{2(s +\frac{2}{q} + \frac{4}{p_0} - 2)} \|P_M e^{\imu t \Delta} f^{\omega} \|_{L^\beta_\omega L^q_t \cL^p_r L^\mu_\theta}^2 \Big)^{\frac{1}{2}}.
	\end{align*}
	Therefore, it is enough to prove 
	\begin{equation}
		\label{eq:ProbabilisticEstimateFirstReduction}
			M^{\frac{2}{q} + \frac{4}{p_0} - 2} \|P_M e^{\imu t \Delta} f^{\omega} \|_{L^\beta_\omega L^q_t \cL^p_r L^\mu_\theta} \lesssim \sqrt{\beta} \|\tilde{P}_M f\|_{L^2}
	\end{equation}
	for all $M \in 2^{\Z}$ with an implicit constant independent of $M$.
	
	We recall from the definition of $f^\omega$ in~\eqref{eq:DefRandomization1} that
	\begin{align*}
		f^\omega = \sum_{M \in 2^{\Z}} f^{M,\omega},
	\end{align*}
	where we use the notation
	\begin{align*}
		f^{M,\omega} &= \sum_{j \in \Z^4} \sum_{k = 0}^\infty \sum_{l = 1}^{N_k} X^{M}_{j,k,l}(\omega) f^{M}_{j,k,l}, \hspace{10em} f^{M}_{j,k,l} = P_j f^{M}_{k,l}, \\
		f^{M}_{k,l}(r\theta) &= a_k M^{-1} r^{-1} b_{k,l}(\theta) \int_0^\infty \hat{c}_{k,l}^M(\rho) J_{k+1}(M r \rho) \rho^2 \dd \rho .
	\end{align*}
	Recall from Subsection~\ref{subsec:Randomization} that the Fourier transform of $f^{M}_{j,k,l}$ is supported in the annulus $\{\frac{M}{2} < |\xi| < 2M\}$. In particular, we have
	\begin{align*}
		P_M e^{\imu t \Delta} f^\omega = P_M e^{\imu t \Delta} (f^{\frac{M}{2},\omega} + f^{M,\omega} + f^{2M, \omega}).
	\end{align*}
	Consequently, for~\eqref{eq:ProbabilisticEstimateFirstReduction} it is enough to prove
	\begin{equation}
		\label{eq:ProbabilisticEstimateSecondReduction}
		M^{\frac{2}{q} + \frac{4}{p_0} - 2} \|e^{\imu t \Delta} f^{M,\omega} \|_{L^\beta_\omega L^q_t \cL^p_r L^\mu_\theta} \lesssim \sqrt{\beta} \|P_M f\|_{L^2}
	\end{equation}
	for every $M \in 2^{\Z}$. Applying Minkowski's inequality and Lemma~\ref{lem:LargeDeviation}, we obtain 
	\begin{equation}
	\label{eq:ProbabilisticEstimateApplyingLargeDeviation}
	 \|e^{\imu t \Delta} f^{M,\omega} \|_{L^\beta_\omega L^q_t \cL^p_r L^\mu_\theta} \lesssim \sqrt{\beta} \|e^{\imu t \Delta} f^{M}_{j,k,l} \|_{L^q_t \cL^p_r L^\mu_\theta l^2_{j,k,l}} \lesssim \sqrt{\beta} \|P_j e^{\imu t \Delta} f^{M}_{k,l} \|_{l^2_{k,l} L^q_t \cL^p_r L^\mu_\theta l^2_{j}}. 
	\end{equation}
	We next employ Lemma~\ref{lem:SquareFunction} with $p_0$ on the right-hand side which yields
	\begin{equation}
	\label{eq:ProbabilisticEstimateSquareFunction}
		\|P_j e^{\imu t \Delta} f^{M}_{k,l} \|_{l^2_{k,l} L^q_t \cL^p_r L^\mu_\theta l^2_{j}} \lesssim \|e^{\imu t \Delta} f^{M}_{k,l} \|_{l^2_{k,l} L^q_t \cL^{p_0}_r L^\mu_\theta}.
	\end{equation}
	We thus obtain~\eqref{eq:ProbabilisticEstimateSecondReduction} from~\eqref{eq:ProbabilisticEstimateApplyingLargeDeviation} and~\eqref{eq:ProbabilisticEstimateSquareFunction} if we show
	\begin{equation}
		\label{eq:ProbabilisticEstimateThirdReduction}
			M^{\frac{2}{q} + \frac{4}{p_0} - 2}\|e^{\imu t \Delta} f^{M}_{k,l} \|_{l^2_{k,l} L^q_t \cL^{p_0}_r L^\mu_\theta} \lesssim \|P_M f\|_{L^2}.
	\end{equation}
	We recall from~\eqref{eq:DefgM} that $g^M = (P_M f)(M^{-1} \cdot)$ and we also rescale $f^M_{k,l}$ to unit frequency by setting $g^M_{k,l} = f^{M}_{k,l}(M^{-1} \cdot)$, i.e.
	\begin{equation}
	\label{eq:DefgMkl}
		 g^M_{k,l}(r \theta) = a_k r^{-1} b_{k,l}(\theta) \int_0^\infty \hat{c}_{k,l}^M(\rho) J_{k+1}(r \rho) \rho^2 \dd \rho.
	\end{equation}
	By scaling, \eqref{eq:ProbabilisticEstimateThirdReduction} is equivalent to
	\begin{equation}
		\label{eq:ProbabilisticEstimateFourthReduction}
		\|e^{\imu t \Delta} g^{M}_{k,l} \|_{l^2_{k,l} L^q_t \cL^{p_0}_r L^\mu_\theta} \lesssim \|g^M\|_{L^2}.
	\end{equation}
	Since $g^{M}_{k,l}$ has unit frequency and thus $e^{\imu t \Delta} g^{M}_{k,l} = \tilde{P}_{2^0} e^{\imu t \Delta} g^{M}_{k,l}$, it is enough to show
	\begin{equation}
		\label{eq:ProbabilisticEstimateFifthReduction}
		\|P_{2^0} e^{\imu t \Delta} g^{M}_{k,l} \|_{l^2_{k,l} L^q_t \cL^{p_0}_r L^\mu_\theta} \lesssim \|g^M\|_{L^2}
	\end{equation}
	in order to prove~\eqref{eq:ProbabilisticEstimateFourthReduction}. Theorem~3.10 in~\cite{SW71} and~\eqref{eq:DefgMkl} imply that we have the representation
	\begin{align*}
		(P_{2^0} e^{\imu t \Delta} g^{M}_{k,l})(r \theta) &= a_k r^{-1} b_{k,l}(\theta) \int_0^\infty \chi_{2^0}(\rho) e^{-\imu t \rho^2} \hat{c}^M_{k,l}(\rho) J_{k+1}(r \rho) \rho^2 \dd \rho \\
		&= a_k b_{k,l}(\theta) T_2^{k+1}(\hat{c}^M_{k,l})(t,r)
	\end{align*}
	with $T_2^{k+1}$ from~\eqref{eq:DefT2kpl1}. Employing property~\eqref{eq:BoundednessGoodFrame} of the good frame, estimate~\eqref{eq:EstimateT2kpl1} for $T_2^{k+1}$, and~\eqref{eq:PlancherelForgM}, we infer
	\begin{align*}
		\|P_{2^0} e^{\imu t \Delta} g^{M}_{k,l} \|_{l^2_{k,l} L^q_t \cL^{p_0}_r L^\mu_\theta} &\lesssim \| b_{k,l}  T_2^{k+1}(\hat{c}^M_{k,l}) \|_{l^2_{k,l} L^q_t \cL^{p_0}_r L^\mu_\theta} \lesssim \| T_2^{k+1}(\hat{c}^M_{k,l}) \|_{l^2_{k,l} L^q_t \cL^{p_0}_r} \\
		&\lesssim \| \hat{c}^M_{k,l}\|_{l^2_{k,l} L^2_\rho} \lesssim \| \hat{c}^M_{k,l}\|_{l^2_{k,l} \cL^2_\rho} \lesssim \|g^M\|_{L^2_x},
	\end{align*}
	i.e.~\eqref{eq:ProbabilisticEstimateFifthReduction}. In view of our reductions, we conclude the assertion of part~(i).
	
	(ii) For the second part we first note that 
		\begin{align*}
			\|\langle \nabla \rangle^{s'} (\ph_i f)\|_{l^2_i \dot{H}^s} \lesssim \|\ph_i f\|_{l^2_i H^{s+s'}} \lesssim \|f\|_{H^{s+s'}}
		\end{align*}
		by Corollary~\ref{cor:CompLPPhysicalSpace}. It is therefore enough to prove
		\begin{align*}
			\| \langle \nabla \rangle^{s'} e^{\imu t \Delta} f^{\tilde{\om}} \|_{L^\beta_{\tilde{\om}} \dot{B}^{s + \frac{2}{q} + \frac{4}{p_0} - 2}_{q,(p,\mu),2}} \lesssim \sqrt{\beta} \|\langle \nabla \rangle^{s'}(\ph_i f)\|_{l^2_i \dot{H}^s},
		\end{align*}
		which follows in the same way as part~(i).
\end{proof}

In view of the definition of the $Y(I)$-norm, we are particularly interested in the following estimates we obtain from the previous proposition. Note that choosing $L^2$ in time allows us to gain $\frac{3}{7}-$ derivatives for the linear flow.
\begin{cor}
	\label{cor:RandomizationInY}
	Let $s \geq 0$, $s' \in \R$, and $0 < \delta \ll 1$. Take $f \in H^s(\R^4)$ and let $f^\omega$ and $f^{\tilde{\omega}}$ denote its randomization from~\eqref{eq:DefRandomization1} and~\eqref{eq:RandomizationFreqAngleSpace}, respectively.
	\begin{enumerate}
		\item \label{it:CorProbabilisticEstimateom}
		There exists a constant $C > 0$ such that
			\begin{align*}
				\|\langle \nabla \rangle^{s'} e^{\imu t \Delta} f^\omega\|_{L^\beta_\om L^2_t \dot{B}^{s + \frac{3}{7} - \frac{2}{7}\delta}_{(p,\mu),2}} + \| \langle \nabla \rangle^{s'} e^{\imu t \Delta} f^\omega \|_{L^\beta_\om \dot{B}^s_{\frac{1}{\delta},(\tilde{p},\mu),2}}\leq C \sqrt{\beta} \| \langle \nabla \rangle^{s'} f\|_{\dot{H}^s}
			\end{align*}
			for all $p \in [\frac{14}{5-\delta},\infty)$, $\tilde{p} \in [\frac{2}{1-\delta},\infty)$, $\mu \in [2,\infty)$, and $\beta \in [1,\infty)$.
			
		\item \label{it:CorProbabilisticEstimatetildeom}
		There exists a constant $C > 0$ such that
			\begin{align*}
				\| \langle \nabla \rangle^{s'} e^{\imu t \Delta} f^{\tilde{\omega}}\|_{L^\beta_{\tilde{\om}} L^2_t \dot{B}^{s + \frac{3}{7} - \frac{2}{7} \delta}_{(p,\mu),2}} + \| \langle \nabla \rangle^{s'} e^{\imu t \Delta} f^{\tilde{\omega}} \|_{L^\beta_{\tilde{\om}} \dot{B}^s_{\frac{1}{\delta},(\tilde{p},\mu),2}}\leq C \sqrt{\beta} \|f\|_{H^{s+s'}}
			\end{align*}
			for all $p \in [\frac{14}{5-\delta},\infty)$, $\tilde{p} \in [\frac{2}{1-\delta},\infty)$, $\mu \in [2,\infty)$, and $\beta \in [1,\infty)$.
	\end{enumerate}
\end{cor}

\begin{proof}
	Note that $L^2(I, \dot{B}^{s + \frac{3}{7} - \frac{2}{7}\delta}_{(p,\mu),2}(\R^4)) = \dot{B}^{s + \frac{3}{7} - \frac{2}{7}\delta}_{2,(p,\mu),2}(I \times \R^4)$ for any time interval $I$. Since $(2, \frac{14}{5-\delta})$ and $(\frac{1}{\delta}, \frac{2}{1-\delta})$ both satisfy~\eqref{eq:ExtendedRange}, the assertions follow from Proposition~\ref{prop:ProbabilisticEstimate} applied with $p_0 = \frac{14}{5-\delta}$ and $\tilde{p}_0 = \frac{2}{1-\delta}$.
\end{proof}

We are now ready to provide the improvement in space-time integrability of the linear flow $e^{\imu t \Delta} f^{\tilde{\om}}$ with data randomized in frequency space, the angular variable and in physical space which originates in the unit-scale decomposition in physical space. We obtain the same set of estimates as for the pure physical-space randomization, see~\cite{S2020, NY2019}. Hence, the additional randomization in frequency space and the angular variable does not impair the improvements of the physical-space randomization.
\begin{prop}
\label{prop:ProbabilisticEstimate2}
	Let $s \in \R$ with $s \geq 0$. Take $q \in [1,\infty)$, $p \in [2,\infty)$, and $\sigma \geq 0$ such that
	\begin{align*}
		2 - \frac{1}{q} - \frac{4}{p} - \sigma > 0.
	\end{align*}
	Pick $f \in H^s(\R^4)$ and let $f^{\tilde{\omega}}$ denote its randomization from~\eqref{eq:RandomizationFreqAngleSpace}. Then there exists a constant $C > 0$ such that
	\begin{align*}
		\|t^\sigma e^{\imu t \Delta} f^{\tilde{\omega}} \|_{L^\beta_{\tilde{\omega}}L^q_{[1,\infty)} \dot{B}^{s}_{p,2}} \leq C \sqrt{\beta} \|f\|_{H^s}
	\end{align*}
	for all $\beta \in [1,\infty)$.
\end{prop}

\begin{proof}
	As $\Omega$ is a probability space, it is enough to show the assertion for $\beta \geq \max\{p,q\} \geq 2$.
	
	We first recall from the definition of $f^{\tilde{\omega}}$ in~\eqref{eq:decf2} that
	\begin{align*}
		f^{\tilde{\omega}} = \sum_{M \in 2^\Z} f^{M, \tilde{\omega}}
	\end{align*}
	with
	\begin{align*}
		f^{M,\tilde{\omega}} &= \sum_{i,j \in \Z^4} \sum_{k = 0}^\infty \sum_{l = 1}^{N_k} X_{i,j,k,l}^M(\tilde{\omega}) \tilde{f}^M_{i,j,k,l}, \hspace{6em} \tilde{f}^{M}_{i,j,k,l} = P_j \tilde{f}^M_{i,k,l}, \\
		\tilde{f}^M_{i,k,l} &= a_k M^{-1} r^{-1} b_{k,l}(\theta) \int_0^\infty \hat{c}^{M,i}_{k,l}(\rho) J_{k+1}(M r \rho) \rho^2 \dd \rho.
	\end{align*}
	We further recall from~\eqref{eq:decphif} that
	\begin{align*}
		P_M(\ph_i f) = \sum_{k = 0}^\infty \sum_{l = 1}^{N_k} \tilde{f}^M_{i,k,l}.
	\end{align*}
	Since the support of the Fourier transform of $f^{M, \tilde{\omega}}$ is contained in $\{\frac{M}{2} < |\xi| < 2M\}$,  we first observe that
	\begin{align*}
		\| t^{\sigma} e^{\imu t \Delta} f^{\tilde{\omega}} \|_{L^\beta_{\tilde{\omega}}L^q_{[1,\infty)} \dot{B}^{s}_{p,2}} \lesssim \Big\| \Big(\sum_{M \in 2^{\Z}} M^{2s} \|t^{\sigma} e^{\imu t \Delta} f^{M, \tilde{\omega}}\|_{L^p_x}^2 \Big)^{\frac{1}{2}} \Big\|_{L^\beta_{\tilde{\omega}}L^q_{[1,\infty)}}.
	\end{align*}
	Applying Minkowski's inequality and Lemma~\ref{lem:LargeDeviation}, we thus infer
	\begin{align}
		\|t^\sigma e^{\imu t \Delta} f^{\tilde{\omega}} \|_{L^\beta_{\tilde{\omega}}L^q_{[1,\infty)} \dot{B}^{s}_{p,2}} &\lesssim \|t^{\sigma} M^s e^{\imu t \Delta} P_j \tilde{f}^M_{i,k,l} \|_{L^q_{[1,\infty)} l^2_M L^p_x l^2_{i,j,k,l}} \nonumber\\
		&\lesssim \|t^\sigma M^s e^{\imu t \Delta} P_j \tilde{f}^M_{i,k,l} \|_{L^q_{[1,\infty)} l^2_M l^2_i L^p_x l^2_{k,l} l^2_j}. \label{eq:ProbabilisticEstimateTime1}
	\end{align}
	We next want to estimate the $l^2_j$-norm. In order to not lose the summability in $l$ and $k$, we cannot simply apply Minkowski's inequality in $p$ and Lemma~\ref{lem:SquareFunction}. Instead, we note that by~\eqref{eq:SquareFunctionEstimate} and Young's inequality we have
	\begin{align}
		\| P_j e^{\imu t \Delta} \tilde{f}^M_{i,k,l} \|_{L^p_x l^2_{k,l} l^2_j} &\lesssim \|  (|\check{\psi}| \ast |e^{\imu t \Delta} \tilde{f}^M_{i,k,l}|^2)^{\frac{1}{2}}\|_{L^p_x l^2_{k,l}}
		= \| |\check{\psi}| \ast \|e^{\imu t \Delta} \tilde{f}^M_{i,k,l}\|_{l^2_{k,l}}^2 \|_{L^{\frac{p}{2}}_x}^{\frac{1}{2}} \nonumber\\
		&\lesssim \|\check{\psi}\|_{L^1_x}^{\frac{1}{2}} \|\|e^{\imu t \Delta} \tilde{f}^M_{i,k,l}\|_{l^2_{k,l}}^2 \|_{L^{\frac{p}{2}}_x}^{\frac{1}{2}} \lesssim \|e^{\imu t \Delta} \tilde{f}^M_{i,k,l}\|_{L^p_x l^2_{k,l}}. \label{eq:ProbabilisticEstimateTime2}
	\end{align}
	Using the definition of $\hat{c}^{M,i}_{k,l}$, Theorem~3.10 in~\cite{SW71}, and rescaling, we further obtain the representation
	\begin{align*}
	e^{\imu t \Delta} \tilde{f}^M_{i,k,l}(r\theta) &= \Big[a_k M^{-1} r^{-1} \int_0^\infty e^{-\imu t M^2 \rho^2} \hat{c}^{M,i}_{k,l}(\rho) J_{k+1}(M r \rho) \rho^2 \dd \rho\Big] \cdot b_{k,l}(\theta) \\
	&=: d^{M}_{i,k,l}(t,r) \cdot b_{k,l}(\theta), \\
		e^{\imu t \Delta} P_M (\ph_i f)(r\theta) &= \sum_{k = 0}^\infty \sum_{l = 1}^{N_k} e^{\imu t \Delta} \tilde{f}^M_{i,k,l}(r \theta) = \sum_{k = 0}^\infty \sum_{l = 1}^{N_k} d^{M}_{i,k,l}(t,r) b_{k,l}(\theta).
	\end{align*}
	In particular, $(d^M_{i,k,l}(t,r))_{k,l}$ are the coefficients of $e^{\imu t \Delta} P_M(\ph_i f)(r \cdot)$ expanded in the orthonormal basis $(b_{k,l})_{k,l}$ of $L^2(S^3)$. Parseval's identity thus yields
	\begin{align*}
		\| d^{M}_{i,k,l}(t,r)\|_{l^2_{k,l}} = \| e^{\imu t \Delta} P_M(\ph_i f)(r \cdot) \|_{L^2_\theta}.
	\end{align*}
	Applying Minkowski's inequality, property~\eqref{eq:BoundednessGoodFrame} of the good frame, and H{\"o}lder's inequality on the sphere, we therefore derive
	\begin{align}
		\|e^{\imu t \Delta} \tilde{f}^M_{i,k,l}\|_{L^p_x l^2_{k,l}} &\lesssim \| d^M_{i,k,l} \cdot b_{k,l} \|_{\cL_r^p l^2_{k,l} L^p_\theta} \lesssim \| d^M_{i,k,l} \|_{\cL_r^p l^2_{k,l}} \lesssim \| e^{\imu t \Delta} P_M(\ph_i f) \|_{\cL^p_r L^2_\theta} \nonumber\\
		&\lesssim \| e^{\imu t \Delta} P_M(\ph_i f) \|_{L^p_x}. \label{eq:ProbabilisticEstimateTime3}
	\end{align}
	The dispersive estimate further implies that
	\begin{align*}
		 \| e^{\imu t \Delta} P_M(\ph_i f) \|_{L^p_x} \lesssim t^{-2 + \frac{4}{p}}  \| P_M(\ph_i f) \|_{L^{p'}_x}
	\end{align*}
	for all $t > 0$. Combining this estimate with~\eqref{eq:ProbabilisticEstimateTime3}, \eqref{eq:ProbabilisticEstimateTime2}, \eqref{eq:ProbabilisticEstimateTime1}, and Corollary~\ref{cor:CompLPPhysicalSpace}, we arrive at
	\begin{align*}
		\|t^\sigma e^{\imu t \Delta} f^{\tilde{\omega}} \|_{L^\beta_{\tilde{\omega}}L^q_{[1,\infty)} \dot{B}^{s}_{p,2}} &\lesssim \|t^\sigma t^{-2 + \frac{4}{p}} M^s P_M(\ph_i f)\|_{L^q_{[1,\infty)} l^2_M l^2_i L_x^{p'}} \\
		&= \|t^{-2 + \frac{4}{p}+\sigma}\|_{L^q_{[1,\infty)}} \| M^s P_M(\ph_i f)\|_{l^2_M l^2_i L_x^{p'}} \lesssim \|f\|_{H^s},
	\end{align*}
	where we also used the assumption on $q$, $p$, and $\sigma$ in the last step.
\end{proof}

While the randomization in frequency space and the angular variable worked together perfectly in Proposition~\ref{prop:ProbabilisticEstimate}, the advantages of the physical-space randomization cannot be used simultaneously in the same way. In order to make use of all the advantages, one thus has to interpolate between the different estimates. This allows us to derive the following estimate, which is the decisive ingredient in the proof of almost sure scattering in Section~\ref{sec:AlmostSureScattering}.
\begin{prop}
\label{prop:L1LinftyEstimate}
	Let $s > \frac{5}{7}$. Take $f \in H^s(\R^4)$ and let $f^{\tilde{\omega}}$ denote its randomization from~\eqref{eq:RandomizationFreqAngleSpace}. Then there is a constant $C > 0$ such that
	\begin{align*}
		\| \nabla e^{\imu t \Delta} f^{\tilde{\omega}}\|_{L^{\beta}_{\tilde{\omega}}L^1_t L^\infty_x} \leq C \sqrt{\beta} \|f\|_{H^s}
	\end{align*}
	for all $\beta \in [1,\infty)$.
\end{prop}

\begin{proof}
	We fix $0 < \delta \ll 1$ such that
	\begin{align*}
		\frac{5}{7} +2 \delta < s
	\end{align*}
	and set $\eta = 2 \delta$. Note that
	\begin{align}
		\| \nabla e^{\imu t \Delta} f^{\tilde{\omega}}\|_{L^{\beta}_{\tilde{\omega}}L^1_{t}L^\infty_x} &\lesssim \|\langle \nabla \rangle^{\delta} \nabla e^{\imu t \Delta} f^{\tilde{\omega}}\|_{L^{\beta}_{\tilde{\omega}}L^1_{t}L^{\frac{8}{\delta}}_x} \nonumber\\
		&\lesssim \|e^{\imu t \Delta} f^{\tilde{\omega}}\|_{L^{\beta}_{\tilde{\omega}}L^1_{t}\dot{B}^{1}_{\frac{8}{\delta},2}} + \|e^{\imu t \Delta} f^{\tilde{\omega}}\|_{L^{\beta}_{\tilde{\omega}}L^1_{t}\dot{B}^{1+\delta}_{\frac{8}{\delta},2}}, \label{eq:EstL1LinftySplit}
	\end{align}
	where we used Sobolev's embedding in the first estimate.
	
	We split the time interval into $\R = (-\infty, 1] \cup (-1,1) \cup [1,\infty)$. Using H{\"o}lder's inequality and Corollary~\ref{cor:RandomizationInY}~\ref{it:CorProbabilisticEstimatetildeom}, we estimate on $(-1,1)$
	\begin{align}
		\|e^{\imu t \Delta} f^{\tilde{\omega}}\|_{L^{\beta}_{\tilde{\omega}}L^1_{(-1,1)}\dot{B}^{1+\delta}_{\frac{8}{\delta},2}} &\lesssim \| e^{\imu t \Delta} f^{\tilde{\omega}}\|_{L^{\beta}_{\tilde{\omega}}L^2_{(-1,1)} \dot{B}^{1+\delta}_{\frac{8}{\delta},2}} \lesssim \| e^{\imu t \Delta} f^{\tilde{\omega}}\|_{L^{\beta}_{\tilde{\omega}}L^2_{t} \dot{B}^{\frac{4}{7} + \frac{9}{7} \delta + \frac{3}{7} - \frac{2}{7} \delta}_{\frac{8}{\delta},2}} \nonumber\\
		&\lesssim \sqrt{\beta} \, \|f\|_{H^{\frac{4}{7} + \frac{9}{7}\delta }} \lesssim \sqrt{\beta}\, \|f\|_{H^s}.
		\label{eq:EstL1Linftycompact}
	\end{align}
	For the estimate on $[1,\infty)$, we first observe that H{\"o}lder's inequality implies
	\begin{align*}
		&\|e^{\imu t \Delta} f^{\tilde{\omega}}\|_{\dot{B}^{1+\delta}_{\frac{8}{\delta},2}} 
		= \Big(\sum_{N \in 2^\Z} N^{2 + 2 \delta}\|P_N e^{\imu t \Delta} f^{\tilde{\omega}}\|_{L^{\frac{8}{\delta}}_x}^2 \Big)^{\frac{1}{2}} \\
		& =\Big(\sum_{N \in 2^\Z} (N^{2 + 2 \delta + \frac{2}{7}(1 + \frac{3\eta}{2-\eta})}\|P_N e^{\imu t \Delta} f^{\tilde{\omega}}\|_{L^{\frac{8}{\delta}}_x}^2)^{\frac{2 - \eta}{3}} (N^{2 + 2 \delta - \frac{4}{7}}\|P_N e^{\imu t \Delta} f^{\tilde{\omega}}\|_{L^{\frac{8}{\delta}}_x}^2)^{\frac{1 + \eta}{3}} \Big)^{\frac{1}{2}} \\
		&\leq \Big(\Big(\sum_{N \in 2^\Z} (N^{\frac{16}{7} + 2 \delta + \frac{6\eta}{14-7\eta}}\|P_N e^{\imu t \Delta} f^{\tilde{\omega}}\|_{L^{\frac{8}{\delta}}_x}^2\Big)^{\frac{2 - \eta}{3}} \Big(\sum_{N \in 2^{\Z}} N^{\frac{10}{7} + 2 \delta}\|P_N e^{\imu t \Delta} f^{\tilde{\omega}}\|_{L^{\frac{8}{\delta}}_x}^2\Big)^{\frac{1 + \eta}{3}} \Big)^{\frac{1}{2}} \\
	&= \|e^{\imu t \Delta} f^{\tilde{\omega}}\|_{\dot{B}^{\frac{8}{7} + \delta + \frac{3\eta}{14-7\eta}}_{\frac{8}{\delta},2}}^{\frac{2-\eta}{3}} \|e^{\imu t \Delta} f^{\tilde{\omega}}\|_{\dot{B}^{\frac{5}{7} + \delta}_{\frac{8}{\delta},2}}^{\frac{1+\eta}{3}}.
	\end{align*}
	Applying H{\"o}lder's inequality in time, we thus obtain
	\begin{align}
		&\|e^{\imu t \Delta} f^{\tilde{\omega}}\|_{L^{\beta}_{\tilde{\omega}}L^1_{[1,\infty)}\dot{B}^{1+\delta}_{\frac{8}{\delta},2}} 
		\lesssim \Big\|t^{\frac{1+\delta}{2}}\|e^{\imu t \Delta} f^{\tilde{\omega}}\|_{\dot{B}^{\frac{8}{7} + \delta + \frac{3\eta}{14-7\eta}}_{\frac{8}{\delta},2}}^{\frac{2-\eta}{3}} \|e^{\imu t \Delta} f^{\tilde{\omega}}\|_{\dot{B}^{\frac{5}{7} + \delta}_{\frac{8}{\delta},2}}^{\frac{1+\eta}{3}} \Big\|_{L^{\beta}_{\tilde{\omega}}L^2_{[1,\infty)}} \nonumber\\
		&\lesssim \|e^{\imu t \Delta} f^{\tilde{\omega}}\|_{L^{\beta}_{\tilde{\omega}}L^2_{[1,\infty)}\dot{B}^{\frac{8}{7} +\delta + \frac{3\eta}{14-7\eta}}_{\frac{8}{\delta},2}}^{\frac{2-\eta}{3}} \|t^{\frac{3}{2} \frac{1+\delta}{1+\eta}} e^{\imu t \Delta} f^{\tilde{\omega}}\|_{L^{\beta}_{\tilde{\omega}}L^2_{[1,\infty)}\dot{B}^{\frac{5}{7} + \delta}_{\frac{8}{\delta},2}}^{\frac{1+\eta}{3}} \nonumber\\
		&\lesssim \|e^{\imu t \Delta} f^{\tilde{\omega}}\|_{L^{\beta}_{\tilde{\omega}}L^2_{t}\dot{B}^{\frac{5}{7} + \frac{9}{7}\delta + \frac{3\eta}{14-7\eta} + \frac{3}{7} - \frac{2}{7}\delta}_{\frac{8}{\delta},2}}^{\frac{2-\eta}{3}} \|t^{\frac{3}{2} \frac{1+\delta}{1+\eta}}e^{\imu t \Delta} f^{\tilde{\omega}}\|_{L^{\beta}_{\tilde{\omega}}L^2_{[1,\infty)}\dot{B}^{\frac{5}{7} + \delta}_{\frac{8}{\delta},2}}^{\frac{1+\eta}{3}}. \label{eq:EstL1LinftyReduction}
	\end{align}
	Recalling that $\eta = 2 \delta \ll 1$, Corollary~\ref{cor:RandomizationInY}~\ref{it:CorProbabilisticEstimatetildeom} now shows that
	\begin{align*}
		\|e^{\imu t \Delta} f^{\tilde{\omega}}\|_{L^{\beta}_{\tilde{\omega}}L^2_{t}\dot{B}^{\frac{5}{7} + \frac{9}{7}\delta + \frac{3\eta}{14-7\eta} + \frac{3}{7} - \frac{2}{7}\delta}_{\frac{8}{\delta},2}} \lesssim \sqrt{\beta} \|f\|_{H^{\frac{5}{7} + \frac{9}{7}\delta + \frac{3\eta}{14-7\eta}}} \lesssim \sqrt{\beta} \|f\|_{H^s},
	\end{align*}
	while Proposition~\ref{prop:ProbabilisticEstimate2} gives
	\begin{align*}
		\|t^{\frac{3}{2} \frac{1+\delta}{1+\eta}} e^{\imu t \Delta} f^{\tilde{\omega}}\|_{L^{\beta}_{\tilde{\omega}}L^2_{[1,\infty)}\dot{B}^{\frac{5}{7} + \delta}_{\frac{8}{\delta},2}} \lesssim \sqrt{\beta} \|f\|_{H^{\frac{5}{7}+\delta}} \lesssim \sqrt{\beta} \|f\|_{H^s}
	\end{align*}
	as
	\begin{align*}
		2 - \frac{1}{2} - \frac{\delta}{2} - \frac{3}{2} \cdot \frac{1+\delta}{1+\eta} = \frac{3}{2} \cdot \frac{\eta - \delta}{1+\eta} - \frac{\delta}{2} = \frac{3}{2} \cdot \frac{\delta}{1+2\delta} - \frac{\delta}{2} > 0.
	\end{align*}
	In view of~\eqref{eq:EstL1LinftyReduction}, we conclude that
	\begin{align}
	\label{eq:EstL1Linftyunbounded}
		\|e^{\imu t \Delta} f^{\tilde{\omega}}\|_{L^{\beta}_{\tilde{\omega}}L^1_{[1,\infty)}\dot{B}^{1+\delta}_{\frac{8}{\delta},2}} 
		\lesssim \sqrt{\beta} \|f\|_{H^s}.
	\end{align}
	By time reversal, we also obtain~\eqref{eq:EstL1Linftyunbounded} with $[1,\infty)$ replaced by $(-\infty,1]$. In combination with~\eqref{eq:EstL1Linftycompact}, we thus arrive at
	\begin{align*}
		\|e^{\imu t \Delta} f^{\tilde{\omega}}\|_{L^{\beta}_{\tilde{\omega}}L^1_{t}\dot{B}^{1+\delta}_{\frac{8}{\delta},2}} 
		\lesssim \sqrt{\beta} \|f\|_{H^s}.
	\end{align*}
	Analogously, we infer
	\begin{align*}
		\|e^{\imu t \Delta} f^{\tilde{\omega}}\|_{L^{\beta}_{\tilde{\omega}}L^1_{t}\dot{B}^{1}_{\frac{8}{\delta},2}} 
		\lesssim \sqrt{\beta} \|f\|_{H^s},
	\end{align*}
	which finishes the proof in view of~\eqref{eq:EstL1LinftySplit}.
\end{proof}

\section{Almost sure local wellposedness}
\label{sec:AlmostSureLocWP}

In this section we prove the local wellposedness of the forced cubic NLS~\eqref{eq:ForcednlSchrdfandf} in the functional framework of Subsection~\ref{subsec:FunctionalFramework} for forcing terms $F \in Y(\R)$. We obtain $e^{\imu t \Delta} f^\omega \in Y(\R)$ almost surely from Corollary~\ref{cor:RandomizationInY} so that Theorem~\ref{thm:locwp} is an immediate consequence.

For the rest of this section, we fix $0 < \delta \ll 1$ and set $\nu = \frac{16}{7} \delta$. The key ingredient in the proof of the local wellposedness theory is the following set of trilinear estimates, which arise from the cubic nonlinearity in the $G(I)$-norm. After frequency localizing the inputs and ordering them by the size of the frequency, we then place those trilinear terms where $v$ appears at highest frequency in the $L^1_t L^2_x$-component of the $G(I)$-norm and estimate them by a combination of deterministic Strichartz estimates and Bernstein's inequality. If the forcing $F$ appears at highest frequency, we place the corresponding trilinear terms in the $L^{\frac{2}{1+2\delta}}_t \cL^{\frac{14}{9+\delta}}_r L^{2}_\theta$-component of the $G(I)$-norm, which gains $\frac{3}{7}-$ derivatives. We then estimate the forcing part of highest frequency in a space which gains another $\frac{3}{7}-$ derivatives in the probabilistic setting.

\begin{lem}
	\label{lem:TrilinearEstimates}
	Let $N, N_1, N_2, N_3 \in 2^{\Z}$ with $N \lesssim N_1$ and $N_3 \leq N_2 \leq N_1$. We then have
	\begin{align}
		N \| v_1  v_2  v_3 \|_{L^1_t L^2_x} &\lesssim \Big(\frac{N}{N_1} \Big) \Big( \frac{N_3}{N_2}\Big)^{\frac{2}{3}} \|v_1\|_{X_{N_1}} \|v_2\|_{X_{N_2}} \|v_3\|_{X_{N_3}}, \label{eq:EstTrilinvvv} \\
		N \| v_1  v_2  F_3 \|_{L^1_t L^2_x} &\lesssim \Big(\frac{N}{N_1} \Big) \Big( \frac{N_3}{N_2}\Big)^{1-2\delta} \|v_1\|_{X_{N_1}} \|v_2\|_{X_{N_2}} \|F_3\|_{Y_{N_3}} , \label{eq:EstTrilinvvF} \\
		N \| v_1  F_2  v_3 \|_{L^1_t L^2_x} &\lesssim \Big(\frac{N}{N_1} \Big) \Big( \frac{N_3}{N_2}\Big)^{\frac{1}{7}} \|v_1\|_{X_{N_1}} \|F_2\|_{Y_{N_2}} \|v_3\|_{X_{N_3}}, \label{eq:EstTrilinvFv} \\
		N \| v_1  F_2 F_3 \|_{L^1_t L^2_x} &\lesssim \Big(\frac{N}{N_1} \Big) \Big( \frac{N_3}{N_2}\Big)^{\frac{4}{7} + \nu} \|v_1\|_{X_{N_1}} \|F_2\|_{Y_{N_2}} \|F_3\|_{Y_{N_3}},  \label{eq:EstTrilinvFF} \\
		N^{\frac{4}{7}+\nu} \| F_1  v_2 v_3 \|_{L^{\frac{2}{1+2\delta}}_t \cL^{\frac{14}{9+\delta}}_r L^2_\theta} &\lesssim \Big(\frac{N}{N_1} \Big)^{\frac{4}{7}+\nu}  \Big( \frac{N_3}{N_2}\Big)^{\delta} \|F_1\|_{Y_{N_1}} \|v_2\|_{X_{N_2}} \|v_3\|_{X_{N_3}} , \label{eq:EstTrilinFvv} \\
		N^{\frac{4}{7}+\nu} \| F_1  v_2 F_3 \|_{L^{\frac{2}{1+2\delta}}_t \cL^{\frac{14}{9+\delta}}_r L^2_\theta} &\lesssim \Big(\frac{N}{N_1} \Big)^{\frac{4}{7}+\nu} \Big( \frac{N_3}{N_2}\Big)^{\frac{1}{2}} \|F_1\|_{Y_{N_1}} \|v_2\|_{X_{N_2}} \|F_3\|_{Y_{N_3}}, \label{eq:EstTrilinFvF} \\
		N^{\frac{4}{7}+\nu} \| F_1  F_2 v_3 \|_{L^{\frac{2}{1+2\delta}}_t \cL^{\frac{14}{9+\delta}}_r L^2_\theta} &\lesssim \Big(\frac{N}{N_1} \Big)^{\frac{4}{7}+\nu}  \Big( \frac{N_3}{N_2}\Big)^{\frac{1}{7}} \|F_1\|_{Y_{N_1}} \|F_2\|_{Y_{N_2}} \|v_3\|_{X_{N_3}}, \label{eq:EstTrilinFFv} \\
		N^{\frac{4}{7}+\nu} \| F_1  F_2 F_3 \|_{L^{\frac{2}{1+2\delta}}_t \cL^{\frac{14}{9+\delta}}_r L^2_\theta} &\lesssim \Big(\frac{N}{N_1} \Big)^{\frac{4}{7}+\nu} \Big( \frac{N_3}{N_2}\Big)^{\frac{1}{7}} \|F_1\|_{Y_{N_1}} \|F_2\|_{Y_{N_2}} \|F_3\|_{Y_{N_3}} \label{eq:EstTrilinFFF}
	\end{align}
	for all $v_j \in X_{N_j}(I \times \R^4)$, $F_j \in Y_{N_j}(I \times \R^4)$, and time intervals $I \subseteq \R$, where all the space-time norms are taken over $I \times \R^4$.
\end{lem}

\begin{proof}
 	We first note that by interpolation between the individual parts of the $Y_N$-norms and H{\"o}lder's inequality on the sphere, we obtain
 	\begin{align}
 		&N_j^{\frac{4}{7} + \nu} \langle N_j \rangle^{4\delta} \|F_j\|_{L^2_t \cL^p_r L^{\mu}_\theta} \lesssim \|F_j\|_{Y_{N_j}},\qquad\langle N_j \rangle^{\frac{1}{7}} \|F_j\|_{L^{\frac{1}{\delta}}_t L^p_x} \lesssim \|F_j\|_{Y_{N_j}} \label{eq:InterploationYNj}
 	\end{align}
 	for all $p \in [\frac{14}{5-\delta},28]$, $\mu \in [1,\frac{4}{\delta}]$, and $j \in \{1,2,3\}$.
	Interpolating between these two estimates next and using H{\"o}lder's inequality on the sphere again, we get
 	\begin{align}
 	\label{eq:CollectionEstimatesYNj}
 		&N_j^{\frac{4}{7}+\nu}\langle N_j \rangle^{2\delta}\|F_j\|_{L^{\frac{2}{1-2\delta}}_t \cL^{p}_r L^{\mu}_\theta} \lesssim \|F_j\|_{Y_{N_j}}
 	\end{align}
 	for all $p \in [\frac{14}{5-\delta},28]$, $\mu \in [1,\frac{1}{\delta}]$, and $j \in \{1,2,3\}$.

	All the estimates in the lemma follow from a combination of H{\"o}lder's inequality, Bernstein's inequality, estimates~\eqref{eq:InterploationYNj} and~\eqref{eq:CollectionEstimatesYNj}, and interpolation between $L^{\frac{1}{\delta}}_tL^{\frac{2}{1-\delta}}_x$ and $L^2_t L^4_x$. For~\eqref{eq:EstTrilinvvv} we infer in this way
	\begin{align*}
		&N \| v_1  v_2  v_3 \|_{L^1_t L^2_x} \lesssim N \|v_1\|_{L^2_t L^4_x} \|v_2\|_{L^3_t L^4_x} \|v_3\|_{L^6_t L^\infty_x} \\
		&\lesssim \Big( \frac{N}{N_1} \Big)\Big(\frac{N_3}{N_2}\Big)^{\frac{2}{3}} N_1 \|v_1\|_{L^2_t L^4_x} N_2 \|v_2\|_{L^3_t L^3_x} N_3 \|v_3\|_{L^6_t L^{\frac{12}{5}}_x} \\
		&\lesssim \Big( \frac{N}{N_1} \Big)\Big(\frac{N_3}{N_2}\Big)^{\frac{2}{3}} \| v_1 \|_{X_{N_1}} \|v_2 \|_{X_{N_2}} \|v_3\|_{X_{N_3}}.
	\end{align*}
	For~\eqref{eq:EstTrilinvvF} we find
	\begin{align*}
		&N \| v_1  v_2  F_3 \|_{L^1_t L^2_x} \lesssim N \|v_1\|_{L^2_t L^4_x} \|v_2\|_{L^{\frac{2}{1-2\delta}}_t L^4_x} \|F_3\|_{L^{\frac{1}{\delta}}_t L^\infty_x} \\
		&\lesssim \Big(\frac{N}{N_1}\Big) N_1 \|v_1\|_{L^2_t L^4_x} N_2 \|v_2\|_{L^{\frac{2}{1-2\delta}}_t L^{\frac{4}{1+2\delta}}_x} \Big( \frac{N_3}{N_2} \Big)^{1-2\delta} \|F_3\|_{L^{\frac{1}{\delta}}_t L^{\frac{4}{1-2\delta}}_x} \\
		&\lesssim \Big(\frac{N}{N_1}\Big)\Big( \frac{N_3}{N_2} \Big)^{1-2\delta} \|v_1\|_{X_{N_1}} \|v_2\|_{X_{N_2}} \|F_3\|_{Y_{N_3}}.
	\end{align*}
	To prove~\eqref{eq:EstTrilinvFv}, we estimate
	\begin{align*}
		&N \| v_1  F_2  v_3 \|_{L^1_t L^2_x} \lesssim N \|v_1\|_{L^{\frac{2}{1-2\delta}}_t L^{\frac{4}{1+2\delta}}_x} \|F_2\|_{L^2_t L^{28}_x} \|v_3\|_{L^{\frac{1}{\delta}}_t L^{\frac{14}{3-7\delta}}_x} \\
		&\lesssim \Big( \frac{N}{N_1} \Big) N_1 \|v_1\|_{L^{\frac{2}{1-2\delta}}_t L^{\frac{4}{1+2\delta}}_x} N_2^{\frac{4}{7}+\nu} \|F_2\|_{L^2_t L^{\frac{28}{4+7\nu}}_x} \Big(\frac{N_3}{N_2}\Big)^{\frac{1}{7}} N_3 \|v_3\|_{L^{\frac{1}{\delta}}_t L^{\frac{2}{1-\delta}}_x} \\
		&\lesssim \Big( \frac{N}{N_1} \Big)\Big(\frac{N_3}{N_2}\Big)^{\frac{1}{7}} \|v_1\|_{X_{N_1}} \|F_2\|_{Y_{N_2}} \|v_3\|_{X_{N_3}}.
	\end{align*}
	Next we show~\eqref{eq:EstTrilinvFF}. Here we infer
	\begin{align*}
		&N \| v_1  F_2  F_3 \|_{L^1_t L^2_x} \lesssim N \|v_1\|_{L^{\frac{2}{1-2\delta}}_t L^{\frac{4}{1+2\delta}}_x} \|F_2\|_{L^2_t L^{\frac{4}{1-2\delta}}_x} \|F_3\|_{L^{\frac{1}{\delta}}_t L^\infty_x} \\
		&\lesssim \Big( \frac{N}{N_1} \Big) N_1 \|v_1\|_{L^{\frac{2}{1-2\delta}}_t L^{\frac{4}{1+2\delta}}_x} N_2^{\frac{4}{7}+\nu} \|F_2\|_{L^2_t L^{\frac{4}{1-2\delta}}_x} \Big(\frac{N_3}{N_2} \Big)^{\frac{4}{7}+\nu} \|F_3\|_{L^{\frac{1}{\delta}}_t L^{\frac{7}{1+4\delta}}_x} \\
		&\lesssim \Big( \frac{N}{N_1} \Big) \Big(\frac{N_3}{N_2} \Big)^{\frac{4}{7}+\nu} \|v_1\|_{X_{N_1}} \|F_2\|_{Y_{N_2}} \|F_3\|_{Y_{N_3}}.
 	\end{align*}
 	
 	To deduce~\eqref{eq:EstTrilinFvv}, we estimate
 	\begin{align*}
 		&N^{\frac{4}{7}+\nu} \| F_1  v_2 v_3 \|_{L^{\frac{2}{1+2\delta}}_t \cL^{\frac{14}{9+\delta}}_r L^2_\theta} 
 		\lesssim N^{\frac{4}{7}+\nu} \|F_1\|_{L^{\frac{2}{1-2\delta}}_t \cL^{\frac{7}{1+11\delta}}_r L^{\frac{2}{3\delta}}_\theta} \|v_2\|_{L^{\frac{1}{\delta}}_t L^{\frac{4}{1-3\delta}}_x} \|v_3\|_{L^{\frac{1}{\delta}}_t L^{\frac{4}{1-3\delta}}_x} \\
 		&\lesssim \Big(\frac{N}{N_1}\Big)^{\frac{4}{7} + \nu} N_1^{\frac{4}{7}+\nu+2\delta} \|F_1\|_{L^{\frac{2}{1-2\delta}}_t \cL^{\frac{7}{1+11\delta}}_r L^{\frac{2}{3\delta}}_\theta} \Big(\frac{N_2}{N_1}\Big)^{\delta} N_2 \|v_2\|_{L^{\frac{1}{\delta}}_t L^{\frac{2}{1-\delta}}_x} \\
 		&\hspace{25em} \cdot \Big(\frac{N_3}{N_1}\Big)^{\delta} N_3 \|v_3\|_{L^{\frac{1}{\delta}}_t L^{\frac{2}{1-\delta}}_x} \\
 		&\lesssim \Big( \frac{N}{N_1} \Big)^{\frac{4}{7} + \nu} \Big(\frac{N_3}{N_2}\Big)^{\delta} \|F_1\|_{Y_{N_1}} \|v_2\|_{X_{N_2}} \|v_3\|_{X_{N_3}}.
 	\end{align*}
 	Next, estimate~\eqref{eq:EstTrilinFvF} follows from
 	\begin{align*}
 		&N^{\frac{4}{7}+\nu} \| F_1  v_2 F_3 \|_{L^{\frac{2}{1+2\delta}}_t \cL^{\frac{14}{9+\delta}}_r L^2_\theta} \\
 		&\lesssim N^{\frac{4}{7}+\nu} \| F_1 \|_{L^{\frac{2}{1-2\delta}}_t \cL^{\frac{28}{7 + 16 \delta}}_r L^{\frac{28}{3 + 14\delta}}_\theta} \|v_2\|_{L^{\frac{1}{\delta}}_t L^{\frac{8}{3-4\delta}}_x} \|F_3\|_{L^{\frac{1}{\delta}} L^{56}_x} \\
 		&\lesssim \Big(\frac{N}{N_1}\Big)^{\frac{4}{7}+\nu}  N_1^{\frac{4}{7} + \nu} \|F_1\|_{L^{\frac{2}{1-2\delta}}_t \cL^{\frac{28}{7 + 16 \delta}}_r L^{\frac{28}{3 + 14\delta}}_\theta} N_2 \|v_2\|_{L^{\frac{1}{\delta}}_t L^{\frac{2}{1-\delta}}_x} \Big(\frac{N_3}{N_2}\Big)^{\frac{1}{2}} \|F_3\|_{L^{\frac{1}{\delta}}_t L^7_x} \\
 		&\lesssim \Big(\frac{N}{N_1}\Big)^{\frac{4}{7}+\nu} \Big(\frac{N_3}{N_2}\Big)^{\frac{1}{2}} \|F_1\|_{Y_{N_1}} \|v_2\|_{X_{N_2}} \|F_3\|_{Y_{N_3}}.
 	\end{align*}
 	Continuing with~\eqref{eq:EstTrilinFFv}, we obtain
 	\begin{align*}
 		&N^{\frac{4}{7}+\nu} \| F_1  F_2 v_3 \|_{L^{\frac{2}{1+2\delta}}_t \cL^{\frac{14}{9+\delta}}_r L^2_\theta} \\
 		&\lesssim N^{\frac{4}{7}+\nu} \| F_1 \|_{L^{\frac{2}{1-2\delta}}_t \cL^{\frac{14}{5-\delta}}_r L^{\frac{14}{3 -2\delta}}_\theta} \| F_2 \|_{L^{\frac{1}{\delta}}_t L^{\frac{14}{1+9\delta}}_x} \|v_3\|_{L^{\frac{1}{\delta}}_t L^{\frac{14}{3-7\delta}}_x} \\
 		&\lesssim \Big(\frac{N}{N_1}\Big)^{\frac{4}{7} + \nu} N_1^{\frac{4}{7} + \nu} \| F_1 \|_{L^{\frac{2}{1-2\delta}}_t \cL^{\frac{14}{5-\delta}}_r L^{\frac{14}{3 -2\delta}}_\theta} N_2^{\frac{1}{7}} \| F_2 \|_{L^{\frac{1}{\delta}}_t L^{\frac{14}{1+9\delta}}_x} \Big(\frac{N_3}{N_2}\Big)^{\frac{1}{7}} N_3 \|v_3\|_{L^{\frac{1}{\delta}}_t L^{\frac{2}{1-\delta}}_x} \\
 		&\lesssim \Big(\frac{N}{N_1}\Big)^{\frac{4}{7} + \nu} \Big(\frac{N_3}{N_2}\Big)^{\frac{1}{7}} \|F_1\|_{Y_{N_1}} \|F_2\|_{Y_{N_2}} \|v_3\|_{X_{N_3}}.
 	\end{align*}
 	Finally, we infer~\eqref{eq:EstTrilinFFF} via
 	\begin{align*}
 		&N^{\frac{4}{7}+\nu} \| F_1  F_2 F_3 \|_{L^{\frac{2}{1+2\delta}}_t \cL^{\frac{14}{9+\delta}}_r L^2_\theta} 
 		\lesssim N^{\frac{4}{7}+\nu} \|F_1\|_{L^{\frac{2}{1-2\delta}}_t \cL^{\frac{14}{5-\delta}}_r L^{\frac{14}{3-2\delta}}_\theta} \|F_2\|_{L^{\frac{1}{\delta}}_t L^{\frac{7}{2}}_x} \|F_3\|_{L^{\frac{1}{\delta}}_t L^{\frac{7}{\delta}}_x} \\
 		&\lesssim \Big(\frac{N}{N_1} \Big)^{\frac{4}{7} + \nu} N_1^{\frac{4}{7}+\nu} \|F_1\|_{L^{\frac{2}	{1-2\delta}}_t \cL^{\frac{14}{5-\delta}}_r L^{\frac{14}{3-2\delta}}_\theta} N_2^{\frac{1}{7}} \|F_2\|_{L^{\frac{1}{\delta}}_t L^{\frac{7}{2}}_x} \Big(\frac{N_3}{N_2}\Big)^{\frac{1}{7}} \|F_3\|_{L^{\frac{1}{\delta}}_t L^{\frac{28}{1+4\delta}}_x} \\
 		&\lesssim \Big(\frac{N}{N_1} \Big)^{\frac{4}{7} + \nu} \Big(\frac{N_3}{N_2}\Big)^{\frac{1}{7}} \|F_1\|_{Y_{N_1}} \|F_2\|_{Y_{N_2}} \|F_3\|_{Y_{N_3}}. \hfill \qedhere
 	\end{align*}
\end{proof}

The above set of trilinear estimates implies the following estimates for the nonlinearity of the forced cubic NLS in our functional framework. They are a key ingredient in the proofs of the local wellposedness theorem and the conditional scattering result for~\eqref{eq:ForcednlSchrdfandf}.
\begin{cor}
	\label{cor:NonlinearEstimate}
	There is a constant $C > 0$ such that
	\begin{align}
		&\| |v + F|^2 (v + F) \|_{G(I)} \leq C (\|v\|_{X(I)}^3 + \|F\|_{Y(I)}^3), \label{eq:NonlinEstCubic}\\
		&\| |v_1 + F|^2 (v_1 + F) - |v_2 + F|^2 (v_2 + F) \|_{G(I)} \label{eq:NonlinEstDifference}\\
		&\hspace{6em} \leq C \|v_1 - v_2\|_{X(I)} (\|v_1\|_{X(I)}^2 + \|v_2\|_{X(I)}^2 + \|F\|_{Y(I)}^2 ),  \nonumber\\
		&\| |u + w + F|^2(u+w+F) - |u|^2 u\|_{G(I)} \label{eq:NOnlinEstPerturbation}\\
		&\hspace{6em} \leq C (\|F\|_{Y(I)}^3 + \|w\|_{X(I)}^3 + \|F\|_{Y(I)} \|u\|_{X(I)}^2 + \|w\|_{X(I)} \|u\|_{X(I)}^2) \nonumber
	\end{align}
	for all $u,v,v_1,v_2,w \in X(I)$, $F \in Y(I)$, and time intervals $I \subseteq \R$.
\end{cor}

\begin{proof}
	All three estimates follow by applying dyadic decomposition and Lemma~\ref{lem:TrilinearEstimates}. The exponential gains in~\eqref{eq:EstTrilinvvv}-\eqref{eq:EstTrilinFFF} allow us to close the estimates via Cauchy-Schwarz and Young's inequality for series.
\end{proof}

We next present the local wellposedness theory for the forced cubic NLS~\eqref{eq:ForcednlSchrdfandf} in the functional framework from Subsection~\ref{subsec:FunctionalFramework}. We show that this problem has a unique solution provided the profile of the initial value and the forcing are sufficiently small. Moreover, we give a blow-up criterion in the $X([t_0, T_+))$-norm, where $T_+$ denotes the maximal time of existence of the solution. In the case $T_+ = \infty$, we also prove that finite $X([t_0, T_+))$-norm implies scattering.
\begin{thm}
	\label{thm:LWPForcedEquation}
	Let $I \subseteq \R$ be a time interval, $t_0 \in I$, $v_0 \in \dot{H}^1(\R^4)$, and $F \in Y(\R)$. There exists $\epsilon > 0$ such that
	\begin{align*}
		\|e^{\imu (t - t_0)\Delta} v_0\|_{X(I)} + \|F\|_{Y(I)} \leq \epsilon
	\end{align*}
	implies that there exists a unique solution $v \in C(I, \dot{H}^1(\R^4)) \cap X(I)$ of~\eqref{eq:ForcednlSchrdfandf}. It extends to a unique solution on the maximal interval of existence $(T_{-}, T_+)$. Moreover,
	\begin{enumerate}
		\item if $T_+ < \infty$, then $\|v\|_{X([t_0,T_+))} = \infty$,
		\item if $T_+ = \infty$ and $\|v\|_{X([t_0,T_+))} < \infty$, then the solution $v$ scatters forward in time, i.e. there exists $v_+ \in \dot{H}^1(\R^4)$ with
		\begin{align*}
			\lim_{t \rightarrow \infty} \|v(t) - e^{\imu t \Delta} v_+\|_{\dot{H}^1} = 0.
		\end{align*}
	\end{enumerate}
	The analogous statements are true for $T_{-}$.
\end{thm}

\begin{proof}
	Set
	\begin{align*}
		B_\epsilon = \{w \in X(I) \colon \|w\|_{X(I)} \leq 2 \epsilon\}
	\end{align*}
	equipped with the metric $d(w_1, w_2) = \|w_1 - w_2\|_{X(I)}$, where $\epsilon > 0$ will be fixed below. We define the fixed point operator
	\begin{align*}
		 \Phi(w)(t) = e^{\imu (t-t_0) \Delta}v_0 \mp \imu \int_{t_0}^t e^{\imu (t-s) \Delta} |F+w|^2(F+w)(s) \dd s
	\end{align*}
	for all $w \in B_\epsilon$. Using~\eqref{eq:StrichartzinXNonlin} and Corollary~\ref{cor:NonlinearEstimate}, we infer
	\begin{align}
	\label{eq:FixedPointSelfMapping}
		\|\Phi(w)\|_{X(I)} &\leq \|e^{\imu (t-t_0) \Delta}v_0\|_{X(I)} + \| |F+w|^2(F+w)\|_{G(I)} \nonumber\\
				&\leq \epsilon + C(\|F\|_{Y(I)}^3 + \|w\|_{X(I)}^3) \leq \epsilon + C \epsilon^3
	\end{align}
	and
	\begin{align}
	\label{eq:FixedPointContraction}
		&\|\Phi(w_1) - \Phi(w_2)\|_{X(I)} \leq C \| |F+w_1|^2 (F+w_1) - |F+w_2|^2 (F + w_2)\|_{G(I)} \nonumber\\
		&\leq C \|w_1 - w_2 \|_{X(I)}(\|w_1\|_{X(I)}^2 + \|w_2\|_{X(I)}^2 + \|F\|_{Y(I)}^2) 
		\leq C \epsilon^2 \|w_1 - w_2\|_{X(I)}
	\end{align}
	for all $w,w_1,w_2 \in B_\epsilon$. Fixing the maximum $C$ of the generic constants on the right-hand sides of~\eqref{eq:FixedPointSelfMapping} and~\eqref{eq:FixedPointContraction}, we set $\epsilon = (\frac{1}{2C})^{\frac{1}{2}}$. Consequently,
	\begin{align*}
		\|\Phi(w)\|_{X(I)} \leq 2 \epsilon, \qquad \|\Phi(w_1) - \Phi(w_2)\|_{X(I)} \leq \frac{1}{2} \|w_1 - w_2\|_{X(I)}
	\end{align*}
	for all $w, w_1, w_2 \in B_\epsilon$, i.e. $\Phi$ is a contractive self-mapping on the complete metric space $B_\epsilon$. The Banach Fixed Point theorem thus yields a unique solution $v$ of~\eqref{eq:ForcednlSchrdfandf} in $B_\epsilon$. Since $\Phi(v) \in C(I, \dot{H}^1(\R^4))$ by Strichartz estimates, we also have $v \in C(I, \dot{H}^1(\R^4))$. Uniqueness without size restriction follows from standard arguments and~\eqref{eq:FixedPointContraction}.
	
	We define the maximal existence times
	\begin{align*}
		 T_+ &= \sup\{T > t_0 \colon \eqref{eq:ForcednlSchrdfandf} \text{ has a solution on } [t_0,T] \}, \\
		 T_{-} &= \inf\{T < t_0 \colon \eqref{eq:ForcednlSchrdfandf} \text{ has a solution on } [T,t_0] \}.
	\end{align*}
	By standard arguments we can extend $v$ to a unique maximal solution on $(T_{-}, T_+)$.
	
	We now turn to the proof of the blow-up condition. Let $T_+ < \infty$. Assume that $\|v\|_{X([t_0, T_+))} < \infty$. Take a sequence $(t_n)_{n \in \N}$ in $[t_0, T_+)$ with $t_n \rightarrow T_+$ as $n \rightarrow \infty$. Note that
	\begin{align*}
		e^{\imu (t - t_n) \Delta} v(t_n) &= e^{\imu t \Delta} v_0 \mp \imu \int_{t_0}^{t_n} e^{\imu (t-s) \Delta} |F+v|^2 (F+v)(s) \dd s \\
		&= v(t) \pm \imu \int_{t_n}^{t} e^{\imu (t-s) \Delta} |F+v|^2 (F+v)(s) \dd s.
	\end{align*}
	We thus obtain from~\eqref{eq:StrichartzinXNonlin}, Corollary~\ref{cor:NonlinearEstimate}, and the dominated convergence theorem
	\begin{align*}
		\|e^{\imu (t - t_n) \Delta} v(t_n)\|_{X([t_n, T_+))} &\leq \|v\|_{X([t_n, T_+))} + \||F+v|^2(F+v)\|_{G([t_n, T_+))} \\
		&\leq \|v\|_{X([t_n, T_+))} + C(\|F\|_{Y([t_n,T_+))}^3 + \|v\|_{X([t_n, T_+))}^3) \longrightarrow 0
	\end{align*}
	as $n \rightarrow \infty$. In particular, we find a time $\overline{t} \in [t_0, T_+)$ such that
	\begin{align*}
		\|e^{\imu (t - \overline{t}) \Delta} v(\overline{t})\|_{X([\overline{t}, T_+))} + \|F\|_{Y([\overline{t},T_+))} \leq \frac{\epsilon}{2}.
	\end{align*}
	The dominated convergence theorem therefore yields $\eta > 0$ with
	\begin{align*}
		\|e^{\imu (t - \overline{t}) \Delta} v(\overline{t})\|_{X([\overline{t}, T_+ +\eta))} + \|F\|_{Y([\overline{t},T_+ + \eta))} \leq \epsilon.
	\end{align*}
	Consequently, we can extend $v$ to a solution on $[t_0, T_+ + \eta)$, which contradicts the definition of $T_+$. We conclude that $\|v\|_{X([t_0,T_+))} = \infty$.
	
	It remains to prove the scattering part. Let $T_+ = \infty$ and $\|v\|_{X([t_0,T_+))} < \infty$. Employing~Lemma~\ref{lem:StrichartzEstimates}, Lemma~\ref{lem:StrichartzEstimatesAv}, and Corollary~\ref{cor:NonlinearEstimate}, we infer that
	\begin{align*}
		v_+ = e^{-\imu t_0 \Delta} v_0 \mp \imu \int_{t_0}^\infty e^{-\imu s \Delta} |F+v|^2 (F+v)(s) \dd s
	\end{align*}
	belongs to $\dot{H}^1(\R^4)$. Then
	\begin{align*}
		\|v(t) - e^{\imu t \Delta} v_+\|_{\dot{H}^1} &\leq \Big\| \int_t^\infty e^{\imu(t-s)\Delta} |F+v|^2 (F+v)(s) \dd s \Big\|_{L^\infty((t,\infty),\dot{H}^1)} \\
		&\lesssim \|F\|_{Y((t,\infty))}^3 + \|v\|_{X((t,\infty))}^3 \longrightarrow 0
	\end{align*}
	as $t \rightarrow \infty$ by the dominated convergence theorem.
	
	The assertions concerning $T_{-}$ are proven analogously.
\end{proof}

Theorem~\ref{thm:locwp} now immediately follows from the local wellposedness theory of the forced cubic NLS.

\noindent
\emph{Proof of Theorem~\ref{thm:locwp}:} Fix $0 < \delta \ll 1$ such that $\frac{1}{7} + 7 \delta \leq s$ and set $\nu = \frac{16}{7}\delta$. Note that $u$ is a solution of~\eqref{eq:nlSchrdfandf} if and only if the nonlinear part $v(t) := u(t) - e^{\imu t \Delta} f^\om$ is a solution of~\eqref{eq:ForcednlSchrdfandf} with $F(t) = e^{\imu t \Delta} f^\om$ and $v_0 = 0$.

We next observe that
\begin{align*}
	\|F\|_{Y(\R)} &\sim \|\langle \nabla \rangle^{4\delta} F\|_{L^2_t \dot{B}^{\frac{4}{7}+\nu}_{(\frac{14}{5-\delta},\frac{4}{\delta}),2}} + \|\langle \nabla \rangle^{4\delta} F\|_{L^2_t \dot{B}^{\frac{4}{7}+\nu}_{(28,\frac{4}{\delta}),2}} + \| \langle \nabla \rangle^{\frac{1}{7}} F\|_{\dot{B}^{0}_{\frac{1}{\delta},\frac{14}{5-\delta},2}} \\
	&\qquad + \| \langle \nabla \rangle^{\frac{1}{7}} F\|_{\dot{B}^{0}_{\frac{1}{\delta},28,2}}.
\end{align*}
Applying Corollary~\ref{cor:RandomizationInY}~\ref{it:ProbabilisticEstimate1om} with regularity parameters $\frac{1}{7} + \frac{18}{7} \delta$ and $4\delta$, we obtain
\begin{align*}
	&\|\langle \nabla \rangle^{4\delta} e^{\imu t \Delta} f^\om\|_{L^2_t \dot{B}^{\frac{4}{7}+\nu}_{(\frac{14}{5-\delta},\frac{4}{\delta}),2}} + \|\langle \nabla \rangle^{4\delta} e^{\imu t \Delta} f^\om\|_{L^2_t \dot{B}^{\frac{4}{7}+\nu}_{(28,\frac{4}{\delta}),2}}   \\
	&\lesssim \sqrt{\beta} \| \langle \nabla \rangle^{4 \delta} f \|_{\dot{H}^{\frac{1}{7}+\frac{18}{7}\delta}} \lesssim \sqrt{\beta} \| f \|_{H^{\frac{1}{7}+7\delta}} \lesssim \sqrt{\beta} \|f\|_{H^s},
\end{align*}
while the same corollary with regularity parameters $0$ and $\frac{1}{7}$ yields
\begin{align*}
	\| \langle \nabla \rangle^{\frac{1}{7}} e^{\imu t \Delta} f^\om\|_{\dot{B}^{0}_{\frac{1}{\delta},\frac{14}{5-\delta},2}} 
 + \| \langle \nabla \rangle^{\frac{1}{7}} e^{\imu t \Delta} f^\om\|_{\dot{B}^{0}_{\frac{1}{\delta},28,2}} \lesssim \sqrt{\beta} \| \langle \nabla \rangle^{\frac{1}{7}} f \|_{L^2} \lesssim \sqrt{\beta} \|f\|_{H^s}.
\end{align*}
 Consequently, $e^{\imu t \Delta} f^\om$ belongs to $Y(\R)$ for almost all $\om \in \Omega$. The assertion of Theorem~\ref{thm:locwp} thus follows from Theorem~\ref{thm:LWPForcedEquation}. \hfill $\qed$

\section{Conditional scattering}
\label{sec:CondScattering}
We now turn our attention to the long time behavior. Contrary to the local theory, the long time behavior of the cubic NLS depends on the sign in front of the nonlinearity. From now on, we only consider the defocusing cubic NLS
\begin{equation}
	\label{eq:CubicNLS}
	\begin{aligned}
	\imu \partial_t u + \Delta u &= |u|^2 u \qquad \text{on } \R \times \R^4, \\
		u(t_0) &= u_0,
	\end{aligned}
\end{equation}
and correspondingly
\begin{equation}
	\label{eq:ForcedCubicNLS}
	\begin{aligned}
	\imu \partial_t v + \Delta v &= |F+v|^2 (F+v) \qquad \text{on } \R \times \R^4, \\
		v(t_0) &= v_0.
	\end{aligned}
\end{equation}

We now proceed as in~\cite{KMV2019} and~\cite{DLM2019} and develop a perturbation theory for~\eqref{eq:ForcedCubicNLS}, which eventually yields the conditional scattering result. We start with a lemma on short-time perturbations.

\begin{lem}
	\label{lem:ShortTimePerturbation}
	Let $I \subseteq \R$ be a time interval and $t_0 \in I$. Then there exist $\epsilon_0, \zeta_0, \eta_0 > 0$ with the following property. If $u_0, v_0 \in \dot{H}^1(\R^4)$ satisfy
	\begin{align*}
		\|u_0 - v_0\|_{\dot{H}^1(\R^4)} \leq \epsilon_0,
	\end{align*}		
	if $u$ is a solution of~\eqref{eq:CubicNLS} on $I$ with $u(t_0) = u_0$ and 
	\begin{align*}
		\|u\|_{X(I)} \leq \zeta
	\end{align*}
	for some $\zeta \in (0,\zeta_0)$, and if $F \in Y(I)$ satisfies
	\begin{align*}
		\|F\|_{Y(I)} \leq \eta
	\end{align*}
	for some $\eta \in (0,\eta_0)$, then there is a unique solution $v \in C(I, \dot{H}^1(\R^4)) \cap X(I)$ of~\eqref{eq:ForcedCubicNLS} and a constant $C \geq 1$ such that
	\begin{equation}
	\label{eq:EstShortTimePerturbation}
		\|u - v\|_{L^\infty_I \dot{H}^1_x} + \|u - v\|_{X(I)} \leq C(\|u_0 - v_0\|_{\dot{H}^1} + \eta).
	\end{equation}
\end{lem}
\begin{proof}
	Due to the local wellposedness theory, it is enough to establish~\eqref{eq:EstShortTimePerturbation} as an a priori estimate. Set $w = v - u$. Then the function $w$ satisfies
	\begin{equation}
		\begin{aligned}
		\imu \partial_t w + \Delta w &= |F+v|^2 (F+v) - |u|^2 u, \\
		w(t_0) &= v_0 - u_0.
	\end{aligned}
	\end{equation}
	Hence, \eqref{eq:StrichartzinXLin}, \eqref{eq:StrichartzinXNonlin}, and Corollary~\ref{cor:NonlinearEstimate} imply
	\begin{align*}
		&\|w\|_{L^\infty_I \dot{H}^1} + \|w\|_{X(I)} \lesssim \|v_0 - u_0\|_{\dot{H}^1} + \||F + u + w|^2(F+u+w) - |u|^2 u\|_{G(I)} \\
		&\lesssim \|v_0 - u_0\|_{\dot{H}^1} + \|F\|_{Y(I)}^3 + \|w\|_{X(I)}^3 + \|F\|_{Y(I)} \|u\|_{X(I)}^2 + \|w\|_{X(I)} \|u\|_{X(I)}^2 \\
		&\lesssim \|v_0 - u_0\|_{\dot{H}^1} + \eta^3 + \|w\|_{X(I)}^3 + \eta \zeta^2 + \zeta^2 \|w\|_{X(I)}.
	\end{align*}
	For $\zeta_0 > 0$ small enough, we thus obtain
	\begin{align*}
		\|w\|_{L^\infty_I \dot{H}^1} + \|w\|_{X(I)} \lesssim \|v_0 - u_0\|_{\dot{H}^1} + \eta + \eta^3 + \|w\|_{X(I)}^3.
	\end{align*}
	If $\epsilon_0 > 0$ and $\eta_0 > 0$ are chosen small enough, a standard continuity argument yields the assertion.
\end{proof}

Due to the time divisibility of the $X(I)$-norm, we can derive a long-time perturbation result from the previous lemma.
\begin{lem}
	\label{lem:LongTimePerturbation}
	Let $I \subseteq \R$ be a time interval, $t_0 \in I$, and  $v_0 \in \dot{H}^1(\R^4)$. Take $A > 0$.  Suppose that the solution $u$ of~\eqref{eq:CubicNLS} with $u(t_0) = v_0$ satisfies
	\begin{align*}
		\|u\|_{X(I)} \leq A.
	\end{align*}
	Then there exists $\eta = \eta(A) > 0$ and $C = C(A) > 0$ such that for every $F \in Y(I)$ satisfying
	\begin{align*}
		\|F\|_{Y(I)} \leq \eta
	\end{align*}
	there is a unique solution $v \in C(I, \dot{H}^1(\R^4)) \cap X(I)$ of~\eqref{eq:ForcedCubicNLS} with
	\begin{equation}
		\label{eq:EstLongTimeStab}
		\|v\|_{X(I)} \leq C(A).
	\end{equation}
\end{lem}
\begin{proof}
	It is again enough to prove~\eqref{eq:EstLongTimeStab} as an a priori estimate because of the local wellposedness theory. We further assume without loss of generality that $t_0 = \inf I$.
	
	Let $\zeta_0,\epsilon_0, \eta_0$ and $C_0$ be the parameters and the constant from Lemma~\ref{lem:ShortTimePerturbation}, respectively, and fix $\zeta \in (0,\zeta_0)$. Then there exist $n = n(A)$ and $t_0 < t_1 < \ldots < t_n = T_+$ such that
	\begin{equation}
		\label{eq:EstuPartitionLongTimeStab}
		\|u\|_{X(I_j)} \leq \zeta
	\end{equation}
	for all $j \in \{1, \ldots, n\}$, where $I_j = (t_{j-1}, t_j)$. We define $\eta = \eta(A)$ by
	\begin{align*}
		 \eta = \min\{\epsilon_0 (2C_0)^{-n}, \eta_0\}.
	\end{align*}
	We show inductively that
	\begin{equation}
		\label{eq:EstDiffInductionHypo}
			\|v-u\|_{L^\infty(I_j, \dot{H}^1)} + \|v-u\|_{X(I_j)} \leq (2C_0)^j \eta
	\end{equation}
	for all $j \in \{1, \ldots, n\}$.
	
	For $j = 1$ this is an immediate consequence of Lemma~\ref{lem:ShortTimePerturbation}. Now assume that we have shown~\eqref{eq:EstDiffInductionHypo} for a number $j \in \{1, \ldots, n-1\}$. Since $(2C_0)^j \eta \leq \epsilon_0$, we can apply Lemma~\ref{lem:ShortTimePerturbation} on $I_{j+1}$ which yields
	\begin{align*}
		\|v - u \|_{L^\infty(I_{j+1}, \dot{H}^1)} + \|v - u \|_{X(I_{j+1})} &\leq C_0 (\|v(t_j) - u(t_j)\|_{\dot{H}^1} + \eta) \\
		&\leq C_0 ((2C_0)^j \eta + \eta) \leq (2C_0)^{j+1} \eta,
	\end{align*}
	i.e.~\eqref{eq:EstDiffInductionHypo} for $j+1$.
	Setting $C = (2C_0)^n \eta + A$ thus finishes the proof.
\end{proof}

We now come to the conditional scattering result for~\eqref{eq:ForcedCubicNLS}. It states that given $F \in Y(\R)$, if the energy of the solution $v$ remains bounded on the maximal interval of existence, then $v$ exists globally and scatters. The key ingredient in the proof is that the deterministic theory for~\eqref{eq:CubicNLS} with initial values in $\dot{H}^1(\R^4)$ provides a bound of a global scattering norm in terms of the energy of the initial value. In view of our perturbation theory, we can derive the conditional scattering result by comparing solutions of the forced equation with the global scattering ones of~\eqref{eq:CubicNLS} with $\dot{H}^1(\R^4)$-initial data.
\begin{prop}
	\label{prop:ConditionalScattering}
	Let $v_0 \in \dot{H}^1(\R^4)$ and $F \in Y([t_0,\infty))$. Let $v$ denote the solution of~\eqref{eq:ForcedCubicNLS} on its maximal interval of existence $[t_0, T_+)$. 
	 Assume that
	\begin{align*}
		A = \sup_{t \in [t_0, T_+)} E(v(t)) < \infty.
	\end{align*}
	We then have $\|v\|_{X([t_0, T_+))} < \infty$. In particular, $T_+ = \infty$ and the solution scatters forward in time, i.e. there is $v_+ \in \dot{H}^1(\R^4)$ such that
	\begin{align*}
		\lim_{t \rightarrow \infty} \|v(t) - e^{\imu t \Delta} v_+\|_{\dot{H}^1} = 0.
	\end{align*}	
	In backward time the analogous statement is true.
\end{prop}
\begin{proof}
	Recall that by the deterministic theory, see~\cite{RV2007, V2012}, for every $u_0 \in \dot{H}^1(\R^4)$ the solution $u$ of~\eqref{eq:CubicNLS} with initial data $u_0$ exists on $\R$ and scatters. We claim that there is a non-decreasing function $K \colon [0,\infty) \rightarrow [0,\infty)$ such that for every initial data $u_0 \in \dot{H}^1(\R^4)$ the global solution $u$ of~\eqref{eq:CubicNLS} satisfies
	\begin{equation}
	\label{eq:EstCubicNLSInX}
		\|u\|_{X(\R)} \leq K(E(u_0)).
	\end{equation}
	To prove this claim, we note that by Theorem~1.1 and Lemma~3.6 in~\cite{RV2007} there exists a non-decreasing function $L \colon [0,\infty) \rightarrow [0,\infty)$ such that 
	\begin{align*}
		 \| \nabla u\|_{L^3_t L^3_x} \leq L(E(u_0)).
	\end{align*}
	Applying~\eqref{eq:StrichartzinXLin}, \eqref{eq:StrichartzinXNonlin}, and Sobolev's inequality, we thus obtain
	\begin{align*}
		\|u\|_{X(\R)} &\lesssim \|u_0\|_{\dot{H}^1} + \||u|^2 u\|_{G(\R)} 
		\lesssim E(u_0)^{\frac{1}{2}} + \Big( \sum_{N \in 2^{\Z}} N^2 \|P_N(|u|^2 u)\|_{L^1_t L^2_x}^2\Big)^{\frac{1}{2}} \\
		&\lesssim  E(u_0)^{\frac{1}{2}} + \|\nabla (|u|^2 u)\|_{L^1_t L^2_x} \lesssim  E(u_0)^{\frac{1}{2}} + \|\nabla u\|_{L^3_t L^3_x} \|u\|_{L^3_t L^{12}_x}^2 \\
		&\lesssim E(u_0)^{\frac{1}{2}} + \|\nabla u\|_{L^3_t L^3_x}^3 \lesssim E(u_0)^{\frac{1}{2}} + L(E(u_0))^3,
	\end{align*}
	implying~\eqref{eq:EstCubicNLSInX}.
	
	Now take $\eta = \eta(K(A)) > 0$ from Lemma~\ref{lem:LongTimePerturbation}. Pick $t_0 < t_1 < \ldots < t_n = T_+$ such that
	\begin{align*}
		\|F\|_{Y(I_j)} \leq \eta
	\end{align*}
	for all $j \in \{1, \ldots, n\}$, where $I_j = (t_{j-1}, t_j)$. Let $u_j$ be the unique global solution of~\eqref{eq:CubicNLS} with initial data $u(t_{j-1}) = v(t_{j-1})$ for every $j \in \{1, \ldots, n\}$. In particular, $u$ exists on $I_j$ and satisfies
	\begin{align*}
		\|u_j\|_{X(I_j)} \leq \|u_j\|_{X(\R)} \leq K(E(v(t_{j-1}))) \leq K(A)
	\end{align*}
	for every $j$, where we used~\eqref{eq:EstCubicNLSInX} and the assumption. Since $\|F\|_{Y(I_j)} \leq \eta(K(A))$, Lemma~\ref{lem:LongTimePerturbation} yields 
	\begin{align*}
		\|v\|_{X(I_j)} \leq C(K(A)).
	\end{align*}		
	Hence,
	\begin{align*}
		\|v\|_{X([t_0,T_+))} \leq \sum_{j = 1}^n \|v\|_{X(I_j)} \leq n C(K(A)) < \infty.
	\end{align*}
	The other assertions of the proposition now follow from Theorem~\ref{thm:LWPForcedEquation}.
\end{proof}

\section{Almost sure scattering}
\label{sec:AlmostSureScattering}

We finally prove the almost sure scattering result for data randomized in frequency space, the angular variable, and physical space. We first formulate conditions on the forcing $F$ which imply scattering for the forced equation
\begin{prop}
	\label{prop:UniformEnergyEstimate}
	Let $0 < \delta \ll 1$, $t_0 \in \R$, $v_0 \in \dot{H}^1(\R^4)$, and let $F \in Y(\R)$ be a solution of the linear Schr{\"o}dinger equation with
	\begin{align*}
		\|\nabla F\|_{L^2_t L^{\frac{14}{5-\delta}}_x} + \|\nabla F\|_{L^2_t L^{\frac{14}{2+\delta}}_x} + \|\nabla F\|_{L^{\frac{2}{1-\delta}}_t L^{\frac{14}{5-\delta}}_x} + \|\nabla F\|_{L^1_t L^\infty_x} + \|F\|_{L^\infty_t L^4_x} < \infty,
	\end{align*}
	where all the space-time norms are taken over $\R \times \R^4$. Let $v$ denote the unique solution of~\eqref{eq:ForcedCubicNLS} with initial data $v_0$ on its maximal interval of existence $(T_{-}, T_+)$. Then
	\begin{align*}
		\sup_{t \in (T_{-}, T_+)} E(v(t)) < \infty.	
	\end{align*}	 
	In particular, $(T_{-}, T_+) = \R$ and the solution $v$ scatters both forward and backward in time, i.e. there exist $v_{\pm} \in \dot{H}^1(\R^4)$ such that
	\begin{align*}
		\lim_{t \rightarrow \pm \infty} \| v(t) - e^{\imu t \Delta} v_{\pm} \|_{\dot{H}^1(\R^4)} = 0.
	\end{align*}
\end{prop}
\begin{proof}
	We only show $T_+ = \infty$ and the scattering forward in time as the proof in the negative time direction works analogously.
	
	A computation as in~\cite{KMV2019} and~\cite{DLM2019} yields
	\begin{align*}
		\partial_t E(v(t)) &= \Re\int_{\R^4}  \partial_t \overline{v}(-\Delta v + |v|^2 v) \dd x \\
		&= - \Re \int_{\R^4}  \partial_t \overline{v}(|F+v|^2(F+v) - |v|^2 v) \dd x, 
	\end{align*}
	where we integrated by parts and used~\eqref{eq:ForcedCubicNLS} after taking the time derivative. We can reformulate this identity as
	\begin{align*}
		\partial_t E(v(t)) &= -\frac{1}{4} \partial_t \int_{\R^4} (|F+v|^4 - |F|^4 - |v|^4) \dd x \\
			&\hspace{6em} + \Re \int_{\R^4} \partial_t \overline{F}(|F+v|^2(F+v) - |F|^2 F) \dd x \\
			&= -\frac{1}{4} \partial_t \int_{\R^4} (|F+v|^4 - |F|^4 - |v|^4) \dd x \\
				&\hspace{6em} -\Im \int_{\R^4} \nabla \overline{F} \cdot \nabla [|F+v|^2 (F+v) - |F|^2 F] \dd x,
	\end{align*}
	where we used that $F$ satisfies $\partial_t F =  \imu \Delta F$ and integrated by parts again.
	Setting
	\begin{align*}
		A_{T_0}(T) = 1+ \sup_{t \in [T_0,T]} E(v(t)),
	\end{align*}
	we thus obtain
	\begin{align}
	\label{eq:EstimateEnergy}
		A_{T_0}(T) &\leq 1 + E(v(T_0)) + \| |F+v|^4 - |F|^4 - |v|^4 \|_{L^\infty_{(T_0,T)}L^1_x} \nonumber\\
		&\hspace{7em} + \| \nabla \overline{F}\cdot \nabla [|F+v|^2 (F+v) - |F|^2 F] \|_{L^1_{(T_0,T)}L^1_x}
	\end{align}
	for all $T_0,T \in (T_{-}, T_+)$ with $T > T_0$. We fix such $T_0$ and $T$ for a moment.
	
	Estimating the boundary terms first, we derive via H{\"o}lder's and Young's inequalities
	\begin{align}
	\label{eq:EstimateBoundaryTerms}
		\| |F+v|^4 - |F|^4 - |v|^4 \|_{L^\infty_tL^1_x} &\leq C_{\kappa} \|F\|_{L^\infty_t L^4_x}^4 + \kappa \|v\|_{L^\infty_t L^4_x}^4 \nonumber\\
		&\leq C_{\kappa} \|F\|_{L^\infty_t L^4_x}^4 + \kappa A_{T_0}(T) 
	\end{align}
	for every $\kappa > 0$, where all the space-time norms are taken over $(T_0,T)\times \R^4$.
	
	It remains to estimate the last summand on the right-hand side of~\eqref{eq:EstimateEnergy}. In the terms appearing there, we can estimate each factor in the same way as its complex conjugate so that we drop the complex conjugation for convenience. Consequently, the last summand on the right-hand side of~\eqref{cor:NonlinearEstimate} is controlled by
	\begin{align}
	\label{eq:EstimatesL1L1}
		\|(\nabla F)^2 F v \|_{L^1_t L^1_x} &\leq \|\nabla F\|_{L^{\frac{2}{1-\delta}}_t L^{\frac{14}{5-\delta}}_x}^2 \|F\|_{L^{\frac{1}{\delta}}_t L^{\frac{28}{1+4\delta}}_x} \|v\|_{L^\infty_t L^4_x} \nonumber\\
		&\lesssim  \|\nabla F\|_{L^{\frac{2}{1-\delta}}_t L^{\frac{14}{5-\delta}}_x}^2 \|F\|_{Y(T_0,T)} A_{T_0}(T), \nonumber\\
		\|(\nabla F)^2 v^2 \|_{L^1_t L^1_x} &\leq \|\nabla F\|_{L^2_t L^{\frac{14}{5-\delta}}_x} \|\nabla F\|_{L^2_t L^{\frac{14}{2+\delta}}_x} \|v\|_{L^\infty_t L^4_x}^2 \nonumber\\
		&\leq \|\nabla F\|_{L^2_t L^{\frac{14}{5-\delta}}_x} \|\nabla F\|_{L^2_t L^{\frac{14}{2+\delta}}_x} A_{T_0}(T),\nonumber\\
		\|(\nabla F) (\nabla v) F^2 \|_{L^1_t L^1_x} &\leq \|\nabla F\|_{L^2_t L^{\frac{14}{5-\delta}}_x} \|\nabla v\|_{L^\infty_t L^2_x} \|F\|_{L^{\frac{2}{1-2\delta}}_t L^{\frac{14}{1-8\delta}}_x} \|F\|_{L^{\frac{1}{\delta}}_t L^{\frac{14}{1+9\delta}}_x}\nonumber\\
		&\lesssim  \|\nabla F\|_{L^2_t L^{\frac{14}{5-\delta}}_x} \|F\|_{L^{\frac{2}{1-2\delta}}_t \dot{B}^{\frac{4}{7}+\nu}_{\frac{14}{3},2}}  \|F\|_{Y(T_0,T)} A_{T_0}(T) \nonumber\\
		&\lesssim  \|\nabla F\|_{L^2_t L^{\frac{14}{5-\delta}}_x} \|F\|_{Y(T_0,T)}^2 A_{T_0}(T),\nonumber\\
		\|(\nabla F) (\nabla v) F v \|_{L^1_t L^1_x} &\leq \|\nabla F\|_{L^2_t L^{\frac{28}{5+16 \delta}}_x} \|\nabla v\|_{L^\infty_t L^2_x} \|F\|_{L^2_t L^{\frac{14}{1-8\delta}}_x} \|v\|_{L^\infty_t L^4_x} \nonumber\\
		&\lesssim  (\|\nabla F\|_{L^2_t L^{\frac{14}{5- \delta}}_x} + \|\nabla F\|_{L^2_t L^{\frac{14}{2+ \delta}}_x} ) \|F\|_{Y(T_0,T)} A_{T_0}(T),\nonumber\\
		\|(\nabla F) (\nabla v) v^2 \|_{L^1_t L^1_x} &\leq \|\nabla F\|_{L^1_t L^\infty_x} \|\nabla v\|_{L^\infty_t L^2_x} \|v\|_{L^\infty_t L^4_x}^2 \leq  \|\nabla F\|_{L^1_t L^\infty_x} A_{T_0}(T),
	\end{align}
	where all space-time norms are taken over $(T_0,T) \times \R^4$. In the derivation of the above estimates, we only used H{\"o}lder's inequality, the embedding $\dot{B}^0_{q,2}(\R^4) \hookrightarrow L^q(\R^4)$ $(q \geq 2)$ in combination with Minkowski's inequality, and~\eqref{eq:InterploationYNj} to estimate
	\begin{align*}
		\|F\|_{L^{\frac{1}{\delta}}_t L^{\frac{28}{1+4\delta}}_x} + \|F\|_{L^{\frac{1}{\delta}}_t L^{\frac{14}{1+9\delta}}_x} \lesssim \|F\|_{Y(T_0,T)},
	\end{align*}		
	as well as $\|P_N F\|_{L^{\frac{14}{1-8\delta}}_x} \lesssim N^{\frac{4}{7} + \nu}\|P_N F\|_{L^{\frac{14}{3}}_x}$ for all $N \in 2^{\Z}$ by Bernstein's inequality combined with~\eqref{eq:InterploationYNj} and~\eqref{eq:CollectionEstimatesYNj} to infer
	\begin{align*}
		\|F\|_{L^{2}_t L^{\frac{14}{1-8\delta}}_x} + \|F\|_{L^{\frac{2}{1-2\delta}}_t L^{\frac{14}{1-8\delta}}_x} \lesssim \|F\|_{Y(T_0,T)},
	\end{align*}
	and interpolation and Young's inequality to get
	\begin{align*}
		\|\nabla F\|_{L^2_t L^{\frac{28}{5+16 \delta}}_x} \lesssim \|\nabla F\|_{L^2_t L^{\frac{14}{5- \delta}}_x} + \|\nabla F\|_{L^2_t L^{\frac{14}{2+ \delta}}_x}.
	\end{align*}
	
	We now fix the maximum $C$ of the implicit constants appearing on the right-hand sides of~\eqref{eq:EstimatesL1L1} and pick $\kappa \in (0,\frac{1}{4})$ such that
	\begin{align*}
		8 C \cdot \kappa < \frac{1}{2}.
	\end{align*}
	The assumptions and the dominated convergence theorem imply that there exist times $t_0 < t_1 < \ldots < t_n < t_{n+1} = T_+$ such that
	\begin{align*}
		\|F\|_{Y(I_j)} + \|\nabla F\|_{L^2_{I_j} L^{\frac{14}{5-\delta}}_x} + \|\nabla F\|_{L^2_{I_j} L^{\frac{14}{2+\delta}}_x} + \|\nabla F\|_{L^{\frac{2}{1-\delta}}_{I_j} L^{\frac{14}{5-\delta}}_x} + \|\nabla F\|_{L^1_{I_j} L^\infty_x} \leq \kappa
	\end{align*}
	for all $j \in \{0, 1, \ldots, n\}$, where $I_j = (t_j, t_{j+1})$ for $j \in \{0, \ldots, n\}$. Combining~\eqref{eq:EstimateEnergy} with~\eqref{eq:EstimateBoundaryTerms} and~\eqref{eq:EstimatesL1L1}, we thus arrive at
	\begin{align*}
		A_{t_j}(t_{j+1}) &\leq 1 + E(v(t_j)) + C_\kappa \|F\|_{L^\infty_{I_j} L^4_x} + 8 C \cdot \kappa A_{t_j}(t_{j+1}), \\
		A_{t_n}(T) &\leq 1 + E(v(t_n)) + C_\kappa \|F\|_{L^\infty_{I_n} L^4_x} + 8 C \cdot \kappa A_{t_n}(T),
	\end{align*}
	implying
	\begin{align*}
		A_{t_j}(t_{j+1}) &\leq 2 + 2 E(v(t_j)) + 2 C_\kappa \|F\|_{L^\infty_{I_j} L^4_x}, \\
		 A_{t_n}(T) &\leq 2 + 2 E(v(t_n)) + 2 C_\kappa \|F\|_{L^\infty_{I_n} L^4_x}
	\end{align*}
	for all $j \in \{0, \ldots, n-1\}$ and $T \in (T_0,T_+)$ by our choice of $\kappa$. Letting $T \rightarrow T_+$, we conclude
	\begin{align*}
		A_{t_n}(t_{n+1}) = A_{t_n}(T_+) \leq  2 + 2 E(v(t_n)) +  2 C_\kappa \|F\|_{L^\infty_{I_n} L^4_x}.
	\end{align*}
	By induction, we get
	\begin{align*}
		A_{t_j}(t_{j+1}) \leq (2^{j+2} - 2)(1 + C_\kappa \|F\|_{L^\infty_t L^4_x}) + 2^{j+1} E(v(t_0)) 
	\end{align*}
	for all $j \in \{0, \ldots, n\}$. We thus arrive at
	\begin{align*}
		\sup_{t \in (t_0, T_+)} E(v(t)) &= \max\{A_{t_j}(t_{j+1}) \colon j \in \{0, \ldots, n\} \} \\
		&= (2^{n+2} - 2)(1 + C_\kappa \|F\|_{L^\infty_t L^4_x}) + 2^{n+1} E(v(t_0))  < \infty.
	\end{align*}
	Arguing analogously on $(T_{-},t_0)$, we obtain
	\begin{align*}
		\sup_{t \in (T_{-}, T_+)} E(v(t)) < \infty.
	\end{align*}
	The remaining assertions of the proposition now follow from Proposition~\ref{prop:ConditionalScattering}.
\end{proof}

Similar to the proof of Theorem~\ref{thm:locwp}, Theorem~\ref{thm:AlmostSureScattering} now immediately follows once we check that $e^{\imu t \Delta} f^{\tilde{\om}}$ satisfies the assumptions on the forcing from Proposition~\ref{prop:UniformEnergyEstimate} almost surely.

\noindent
\emph{Proof of Theorem~\ref{thm:AlmostSureScattering}:} Fix $0 < \delta \ll 1$ such that $\frac{5}{7} + 2\delta \leq s$. Recall that $u$ is a solution of~\eqref{eq:CubicNLS} if and only if $v(t) := u(t) - e^{\imu t \Delta} f^{\tilde{\om}}$ solves~\eqref{eq:ForcedCubicNLS} with forcing $F(t) = e^{\imu t \Delta} f^{\tilde{\om}}$ and initial data $v_0 = 0$.

To check the assumptions of Proposition~\ref{prop:UniformEnergyEstimate}, we first note that analogous to  the proof of Theorem~\ref{thm:locwp}, Corollary~\ref{cor:RandomizationInY}~\ref{it:CorProbabilisticEstimatetildeom} implies $e^{\imu t \Delta} f^{\tilde{\om}} \in Y(\R)$ for almost all $\tilde{\om} \in \Omega$.

Since $(2, \frac{14}{5-\delta})$ and $(\frac{2}{1-\delta}, \frac{14}{5-\delta})$ satisfy~\eqref{eq:ExtendedRange} and $\frac{14}{2+\delta} \geq \frac{14}{5-\delta}$, we can apply Proposition~\ref{prop:ProbabilisticEstimate}~\ref{it:ProbabilisticEstimate1tildeom} with $(q,p_0) = (2, \frac{14}{5-\delta})$ and $(q,p_0) = (\frac{2}{1-\delta}, \frac{14}{5-\delta})$, respectively, to infer
\begin{align}
	&\|\nabla  e^{\imu t \Delta} f^{\tilde{\om}} \|_{L^\beta_{\tilde{\om}}L^2_t L^{\frac{14}{5-\delta}}_x} 
		\lesssim \|e^{\imu t \Delta} f^{\tilde{\om}}\|_{L^\beta_{\tilde{\om}}\dot{B}^1_{2,\frac{14}{5-\delta},2}}  \lesssim \sqrt{\beta} \|f\|_{H^{\frac{4}{7}+\frac{2}{7}\delta}} \lesssim \sqrt{\beta} \|f\|_{H^s}, \nonumber\\
		 &\|\nabla  e^{\imu t \Delta} f^{\tilde{\om}} \|_{L^\beta_{\tilde{\om}}L^2_t L^{\frac{14}{2+\delta}}_x} \lesssim \|e^{\imu t \Delta} f^{\tilde{\om}}\|_{L^\beta_{\tilde{\om}}\dot{B}^1_{2,\frac{14}{2+\delta},2}} \lesssim \sqrt{\beta} \|f\|_{H^{\frac{4}{7}+\frac{2}{7}\delta}}\lesssim \sqrt{\beta} \|f\|_{H^s}, \nonumber\\
		 &\|\nabla  e^{\imu t \Delta} f^{\tilde{\om}} \|_{L^\beta_{\tilde{\om}}L^{\frac{2}{1-\delta}}_t L^{\frac{14}{5-\delta}}_x} 
		\lesssim \|e^{\imu t \Delta} f^{\tilde{\om}}\|_{L^\beta_{\tilde{\om}}\dot{B}^1_{\frac{2}{1-\delta},\frac{14}{5-\delta},2}}  \lesssim \sqrt{\beta} \|f\|_{H^{\frac{4}{7}+\frac{9}{7}\delta}}\lesssim \sqrt{\beta} \|f\|_{H^s} \label{eq:CollectionEstimatesNablaFlow}
\end{align}
for all $\beta \in [1,\infty)$. Proposition~\ref{prop:L1LinftyEstimate} yields
\begin{equation}
\label{eq:EstimateNablaFlowL1Linfty}
	\|\nabla  e^{\imu t \Delta} f^{\tilde{\om}} \|_{L^\beta_{\tilde{\om}}L^1_t L^{\infty}_x} \lesssim \sqrt{\beta} \| f \|_{H^s}
\end{equation}
for all $\beta \in [1,\infty)$. For the remaining $L^\infty_t L^4_x$-estimate, we exploit that $e^{\imu t \Delta} f^{\tilde{\om}}$ is a solution of the linear Schr{\"o}dinger equation so that we obtain after applying Sobolev's embedding
\begin{align*}
	\|e^{\imu t \Delta} f^{\tilde{\om}}\|_{L^\infty_t L^4_x} \lesssim \|\langle \partial_t \rangle ^{2\delta} e^{\imu t \Delta} f^{\tilde{\om}}\|_{L^{\frac{1}{\delta}} L^4_x} \lesssim \|\langle \nabla \rangle ^{4\delta} e^{\imu t \Delta} f^{\tilde{\om}}\|_{L^{\frac{1}{\delta}} L^4_x} \lesssim \|e^{\imu t \Delta} f^{\tilde{\om}}\|_{Y(\R)}.
\end{align*}
Combined with~\eqref{eq:CollectionEstimatesNablaFlow} and~\eqref{eq:EstimateNablaFlowL1Linfty}, we conclude that for almost every $\tilde{\omega} \in \Omega$, $e^{\imu t \Delta} f^{\tilde{\om}}$ belongs to $Y(\R)$ and satisfies the assumptions on the forcing $F$ from Proposition~\ref{prop:UniformEnergyEstimate}. The assertion of Theorem~\ref{thm:AlmostSureScattering} thus follows from Proposition~\ref{prop:UniformEnergyEstimate}. \hfill $\qed$

\vspace{1em}

\textbf{Acknowledgment.} I want to thank Sebastian Herr for helpful discussions.

Funded by the Deutsche Forschungsgemeinschaft (DFG, German Research Foundation) – SFB 1283/2 2021 – 317210226.

\bibliographystyle{abbrv}
\bibliography{SupercriticalCubicNLS} 

\begin{thebibliography}{10}

\bibitem{BOP2015I}
{\'A}.~B{\'e}nyi, T.~Oh, and O.~Pocovnicu.
\newblock On the probabilistic {C}auchy theory of the cubic nonlinear
  {S}chr{\"o}dinger equation on $\mathbb{R}^d$, $d \geq 3$.
\newblock {\em Trans. Amer. Math. Soc. Ser. B}, 2:1--50, 2015.

\bibitem{BOP2015II}
{\'A}.~B{\'e}nyi, T.~Oh, and O.~Pocovnicu.
\newblock Wiener randomization on unbounded domains and an application to
  almost sure well-posedness of {NLS}.
\newblock In {\em Excurions in harmonic analysis}, volume~4 of {\em Appl.
  Numer. Harmon. Anal.}, pages 3--25. Birkh{\"a}user/Springer, Cham, 2015.

\bibitem{BOP2019}
{\'A}.~B{\'e}nyi, T.~Oh, and O.~Pocovnicu.
\newblock Higher order expansions for the probabilistic local {C}auchy theory
  of the cubic nonlinear {S}chr{\"o}dinger equation on $\mathbb{R}^3$.
\newblock {\em Trans. Amer. Math. Soc. Ser. B}, 6:114--160, 2019.

\bibitem{B1994}
J.~Bourgain.
\newblock Periodic nonlinear {S}chr{\"o}dinger equation and invariant measures.
\newblock {\em Comm. Math. Phys.}, 166:1--26, 1994.

\bibitem{B1996}
J.~Bourgain.
\newblock Invariant measures for the {2D}-defocusing nonlinear
  {S}chr{\"o}dinger equation.
\newblock {\em Comm. Math. Phys.}, 176:421--445, 1996.

\bibitem{B2019}
J.~Brereton.
\newblock Almost sure local well-posedness for the supercritical quintic {NLS}.
\newblock {\em Tunis. J. Math.}, 1:427--453, 2019.

\bibitem{Br2019}
B.~Bringmann.
\newblock Almost sure scattering for the energy critical nonlinear wave
  equation.
\newblock To appear in \emph{Amer. J. Math.}, \texttt{arXiv:1812.10187}.

\bibitem{Br2020}
B.~Bringmann.
\newblock Almost-sure scattering for the radial energy-critical nonlinear wave
  equation in three dimensions.
\newblock {\em Anal. PDE}, 13:1011--1050, 2020.

\bibitem{BK2019}
N.~Burq and J.~Krieger.
\newblock Randomization improved {S}trichartz estimates and global
  well-posedness for supercritical data.
\newblock Arxiv preprint, \texttt{arXiv:1902.06987}.

\bibitem{BL2013}
N.~Burq and G.~Lebeau.
\newblock Injections de {S}obolev probabilistes et applications.
\newblock {\em Ann. Sci. {\'E}c. Norm. Sup{\'e}r. (4)}, 46:917--962, 2013.

\bibitem{BL2014}
N.~Burq and G.~Lebeau.
\newblock Probabilistic {S}obolev embeddings, applications to eigenfunctions
  estimates.
\newblock In {\em Geometric and spectral analysis}, volume 630 of {\em Contemp.
  Math.}, pages 307--318. Amer. Math. Soc., Providence, RI, 2014.

\bibitem{BT2008I}
N.~Burq and N.~Tzvetkov.
\newblock Random data {C}auchy theory for supercritical wave equations {I}:
  local theory.
\newblock {\em Invent. Math.}, 173:449--475, 2008.

\bibitem{BT2008II}
N.~Burq and N.~Tzvetkov.
\newblock Random data {C}auchy theory for supercritical wave equations {II}: a
  global existence result.
\newblock {\em Invent. Math.}, 173:477--496, 2008.

\bibitem{C2021}
N.~Camps.
\newblock Scattering for the cubic {S}chr{\"o}dinger equation in {3D} with
  randomized radial initial data.
\newblock Arxiv preprint, \texttt{arXiv:2110.10752}.

\bibitem{CCT2003}
M.~Christ, J.~Colliander, and T.~Tao.
\newblock Ill-posedness for nonlinear {S}chr{\"o}dinger and wave equations.
\newblock \texttt{arXiv:math/0311048}.

\bibitem{CK2001}
M.~Christ and A.~Kiselev.
\newblock Maximal functions associated to filtrations.
\newblock {\em J. Funct. Anal.}, 179:409--425, 2001.

\bibitem{CKSTT2008}
J.~Colliander, M.~Keel, G.~Staffilani, H.~Takaoka, and T.~Tao.
\newblock Global well-posedness and scattering for the energy-critical
  nonlinear {S}chr{\"o}dinger equation in $\mathbb{R}^3$.
\newblock {\em Ann. of Math. (2)}, 167:767--865, 2008.

\bibitem{dS2013}
A.-S. de~Suzzoni.
\newblock Large data low regularity scattering results for the wave equation on
  the {E}uclidean space.
\newblock {\em Comm. Partial Differential Equations}, 38(1):1--49, 2013.

\bibitem{D2012}
Y.~Deng.
\newblock Two-dimensional nonlinear {S}chr{\"o}dinger equation with random
  radial data.
\newblock {\em Anal. PDE}, 5(5):913--960, 2012.

\bibitem{DLM2019}
B.~Dodson, J.~Lührmann, and D.~Mendelson.
\newblock Almost sure local well-posedness and scattering for the {4D} cubic
  nonlinear {S}chrödinger equation.
\newblock {\em Adv. Math.}, 347:619--676, 2019.

\bibitem{DLM2020}
B.~Dodson, J.~L{\"u}hrmann, and D.~Mendelson.
\newblock Almost sure scattering for the {4D} energy-critical defocusing
  nonlinear wave equation with radial data.
\newblock {\em Amer. J. Math.}, 142:475--504, 2020.

\bibitem{Guo2016}
Z.~Guo.
\newblock Sharp spherically averaged {S}trichartz estimates for the
  {S}chr{\"o}dinger equation.
\newblock {\em Nonlinearity}, 29:1668--1686, 2016.

\bibitem{GLNW2014}
Z.~Guo, S.~Lee, K.~Nakanishi, and C.~Wang.
\newblock Generalized {S}trichartz estimates and scattering for {3D} {Z}akharov
  system.
\newblock {\em Comm. Math. Phys.}, 331:239--259, 2014.

\bibitem{KT1998}
M.~Keel and T.~Tao.
\newblock Endpoint {S}trichartz estimates.
\newblock {\em Amer. J. Math.}, 120:955--980, 1998.

\bibitem{KMV2019}
R.~Killip, J.~Murphy, and M.~Visan.
\newblock Almost sure scattering for the energy-critical {NLS} with radial data
  below ${H}^1(\mathbb{R}^4)$.
\newblock {\em Comm. Partial Differential Equations}, 44(1):51--71, 2019.

\bibitem{LM2014}
J.~L{\"u}hrmann and D.~Mendelson.
\newblock Random data {C}auchy theory for nonlinear wave equations of
  power-type on {$\R^3$}.
\newblock {\em Comm. Partial Differential Equations}, 39:2262--2283, 2014.

\bibitem{M2019}
J.~Murphy.
\newblock Random data final-state problem for the mass-subcritical {NLS} in
  {$L^2$}.
\newblock {\em Proc. Amer. Math. Soc.}, 147:339--350, 2019.

\bibitem{NY2019}
K.~Nakanishi and T.~Yamamoto.
\newblock Randomized final-data problem for systems of nonlinear
  {S}chr{\"o}dinger equations and the {G}ross-{P}itaevskii equation.
\newblock {\em Math. Res. Lett.}, 26:253--279, 2019.

\bibitem{OOP2019}
T.~Oh, M.~Okamoto, and O.~Pocovnicu.
\newblock On the probabilistic well-posedness of the nonlinear
  {S}chr{\"o}dinger equations with non-algebraic nonlinearities.
\newblock {\em Discrete Contin. Dyn. Syst.}, 39:3479--3520, 2019.

\bibitem{OP2016}
T.~Oh and O.~Pocovnicu.
\newblock Probabilistic global well-posedness of the energy-critical defocusing
  quintic nonlinear wave equation on {$\R^3$}.
\newblock {\em J. Math. Pures Appl.}, 105:342--366, 2016.

\bibitem{P2017}
O.~Pocovnicu.
\newblock Almost sure global well-posedness for the energy-critical defocusing
  nonlinear wave equation on {$\R^d$}, $d = 4$ and $5$.
\newblock {\em J. Eur. Math. Soc. (JEMS)}, 19:2521--2575, 2017.

\bibitem{PRT2014}
A.~Poiret, D.~Robert, and L.~Thomann.
\newblock Probabilistic global well-posedness for the supercritical nonlinear
  harmonic oscillator.
\newblock {\em Anal. PDE}, 7:997--1026, 2014.

\bibitem{RV2007}
E.~Ryckman and M.~Visan.
\newblock Global well-posedness and scattering for the defocusing
  energy-critical nonlinear {S}chr{\"o}dinger equation in $\mathbb{R}^{1+4}$.
\newblock {\em Amer. J. Math.}, 129:1--60, 2007.

\bibitem{SSW2021}
J.~Shen, A.~Soffer, and Y.~Wu.
\newblock Almost sure well-posedness and scattering of the {3D} cubic nonlinear
  {S}chr{\"o}dinger equation.
\newblock Arxiv preprint, \texttt{arXiv:2110.11648}.

\bibitem{Swave2021}
M.~Spitz.
\newblock On the almost sure scattering for the energy-critical cubic nonlinear
  wave equation with supercritical data.
\newblock In preparation.

\bibitem{S2020}
M.~Spitz.
\newblock Randomized final-state problem for the {Z}akharov system in dimension
  three.
\newblock {\em Comm. Partial Differential Equations}, 2021.
\newblock DOI: 10.1080/03605302.2021.1983595.

\bibitem{SW71}
E.~M. Stein and G.~Weiss.
\newblock {\em {I}ntroduction to {F}ourier {A}nalysis on {E}uclidean {S}paces}.
\newblock Princeton University Press, Princeton, New Jersey, 1971.

\bibitem{T2000}
T.~Tao.
\newblock Spherically averaged endpoint {S}trichartz estaimtes for the
  two-dimensional {S}chr{\"o}dinger equation.
\newblock {\em Comm. Partial Differential Equations}, 25:1471--1485, 2000.

\bibitem{V2012}
M.~Visan.
\newblock Global well-posedness and scattering for the defocusing cubic
  nonlinear {S}chr{\"o}dinger equation in four dimensions.
\newblock {\em Int. Math. Res. Not.}, 2012(5):1037--1067, 2012.

\end{thebibliography}
 
\end{document}